\documentclass[a4paper,11pt,reqno]{amsart}

\usepackage{amsthm, mathrsfs,amssymb,amsmath,}
\usepackage{enumerate}
\usepackage{mathtools}
\makeatletter
\@namedef{subjclassname@2010}{%
	\textup{2010} Mathematics Subject Classification}
\makeatother

\usepackage{hyperref}
\allowdisplaybreaks
\numberwithin{equation}{section}

\usepackage{graphicx}
\usepackage{enumitem}

\newtheorem{theorem}{Theorem}[section]
\newtheorem{lemma}[theorem]{Lemma}

\theoremstyle{definition}
\newtheorem{definition}[theorem]{Definition}
\newtheorem{example}[theorem]{Example}

\newtheorem{proposition}[theorem]{Proposition}
\theoremstyle{remark}
\newtheorem{remark}[theorem]{Remark}
\newtheorem{corollary}[theorem]{Corollary}

\numberwithin{equation}{section}

\usepackage{fancyhdr}
\begin{document}
	
	\title{ Fractal dimensions of fractal transformations and Quantization dimensions for bi-Lipschitz mappings}
	

	
	\author{Manuj Verma}
	\address{Department of Mathematics,  Indian Institute of Technology Delhi, New Delhi, India 110016}
	\email{mathmanuj@gmail.com}
	\author{Amit Priyadarshi}
	\address{Department of Mathematics, Indian Institute of Technology Delhi, New Delhi, India 110016}
	
	\email{priyadarshi@maths.iitd.ac.in}
	\author{Saurabh Verma}
	\address{Department of Applied Sciences,  Indian Institute of information Technology Allahabad, Prayagraj, India 211015 }
	\email{saurabhverma@iiita.ac.in}
	
	
	\date{}

	\begin{abstract}
		In this paper, we study the fractal dimension of the graph of a fractal transformation and also determine the quantization dimension of a probability measure supported on the graph of the fractal transformation. Moreover, we estimate the quantization dimension of the invariant
		measures corresponding to a weighted iterated function system consisting of bi-Lipschitz mappings under the strong open set condition.
	\end{abstract}
	
	\maketitle

	
	
	\section{INTRODUCTION}

	In Fractal Geometry, Iterated Function Systems (IFSs) play an important role. They are commonly used to generate fractals. In fact, most of the fractals are the attractors of some IFS. There are many literature available on the study of IFS, its attractor and fractal dimension of attractor, see, for more details, \cite{Fal,MF2,JH,Nussbaum1,AS}. In 2009, Barnsley \cite{MF7} introduced the idea of fractal transformation. Basically, fractal transformation is a map between the attractor of one IFS to the attractor of another IFS. After that, Barnsley et al. \cite{Igu} discussed many applications of fractal transformation.
	In 2014, Barnsely and his collaborators \cite{MF1} determined some conditions under which fractal transformation is measure preserving.  In 2016, Bandt et al. \cite{Bandt} proved that under some conditions fractal transformation becomes homeomorphism. In 2018, Vince \cite{And} showed that we can extend fractal transformation from the non-empty attractor to whole space and described that under some conditions fractal transformation is the attractor of some IFS, which is constructed from the given IFSs. 
	\par
	The quantization dimension is one of the most important thing in the quantization theory.  In 1963, Zador \cite{ZADOR} was the  first, who introduced the term quantization dimension and also discussed some properties of this dimension. In 2002, Graf and Luschgy \cite{GL2} gave a formula of the quantization dimension of the self-similar measures under the open set condition (OSC). After that, Lindsay et al. \cite{Lindsay} generalized the result of Graf and Luschgy \cite{GL1} for the self-conformal measures and described some connection between quantization theory and fractal geometry. The quantization dimension is also related to some other known fractal dimensions like the Hausdorff dimension and the box-counting  dimension, see, for more details \cite{GL1}. In 2010, Roychowdhury \cite{R1} obtained the quantization dimension of the Moran measures. After that, Roychowdhury \cite{R2} also determined the quantization dimension of a Borel probability measure supported on the attractor of the bi-Lipschitz IFS by taking some conditions on the bi-Lipschitz constant and the IFS satisfying SOSC. In 2021, Roychowdhury and Selmi \cite{R4} estimated bounds of the quantization dimension of the invariant measure generated by the hyperbolic recurrent IFS under the strong separation condition. In this paper, we determine bounds of the quantization dimension of the invariant Borel probability measures supported on the attractor of a general class of bi-Lipschitz IFSs under the SOSC. This result also generalizes the result of Graf and Luschgy \cite{GL1}. We also estimate the quantization dimension for the invariant Borel probability measures supported on the graph of the fractal transformation. Here, we discuss some dimensional results for the graph of the fractal transformation. 
	
	\par
	The paper is organized as follows. In upcoming Section \ref{se2}, we discuss some preliminary results and the required definitions for forthcoming section. In Section \ref{se3}, we give some results on the fractal transformation and the product of two IFSs. Firstly, we determine some results related to the product of two IFSs. Next, we obtain the bounds of the Hausdorff dimension of the graph of the fractal transformation under some conditions. After that, we determine a relation between invariant measures of two IFS and its product IFS. We also give a relation the between the quantization dimension of these invariant measures. In the last section, we provide the bounds of the quantization dimension of the invariant probability measures  corresponding to bi-Lipschitz IFS under the SOSC and also give bounds of the quantization dimension of the invariant measure supported on the graph of the fractal transformation.

	\section{preliminaries}\label{se2}
	\begin{definition}
		Let $F$ be a subset of a metric space $(Y,\rho)$. The Hausdorff dimension of $F$ is defined as follows
		$$ \dim_H{F}=\inf\{\beta>0: \text{for every}~\epsilon>0,~\text{there is a countable cover}~~ \{V_i\}~\text{of}~F~\text{ with}\sum |V_i|^\beta<\epsilon \},$$
		where $|V_i|$ denotes the diameter of $V_i.$  
		
	\end{definition}
	\begin{definition}
		The box dimension of a non-empty bounded subset $F$ of a metric space $(Y,\rho)$ is defined as
		$$\dim_{B}F=\lim_{\delta \to 0}\frac{\log{N_{\delta}(F)}}{-\log\delta},$$
		where $N_\delta(F)$ is the minimum number of sets of diameter $\delta>0$ that can cover $F,$ provided the limit exists. 
		If this limit does not exist, then limsup and liminf are known as the upper and the lower box dimension and are denoted by $\overline{\dim}_B(F)$ and $\underline{\dim}_B(F)$,  respectively.
	\end{definition}

	Let $(Y, \rho)$ be a complete metric space, and we denote the family of all nonempty compact subsets of $Y$ by $H(Y)$ . For any $A,B\in H(Y)$, we define the Hausdorff metric by 
	$$h(A,B) = \inf\{\delta>0 : A\subset {B}_\delta~~\text{and}~~ B \subset {A}_\delta \} ,$$
	where ${A}_\delta$ and ${B}_\delta$ denote the $\delta $-neighbourhoods of sets $A$ and $B$, respectively. It is well-known that $(H(Y),h)$ is a complete metric space. \par
	A map $\theta: (Y,\rho) \to (Y,\rho) $  is called a contraction if there exists a constant  $c<1$ such that 
	$$\rho(\theta(a),\theta(b)) \le c~~\rho(a,b),~~\forall~~~ a , b \in Y.$$
	\begin{definition}
		The system $\mathcal{I}=\big\{(Y,\rho); \theta_1,\theta_2,\dots,\theta_N \big\}$ is called an iterated function system (IFS), if each $\theta_i$ is a contraction self-map on $Y$ for $i\in \{1,2,\dots,N\}$.
	\end{definition}
	\begin{definition}
		The system $\{(Y,\rho); \theta_1,\theta_2,\dots,\theta_N; p_1,p_2,\cdots,p_N\}$ is called a weighted iterated function system (WIFS) if $\{(Y,\rho):\theta_1,\theta_2,\dots,\theta_N\}$ is an IFS with probability vector $(p_1,p_2,\cdots,p_N)$.
	\end{definition}
	\begin{remark}
		$(p_1,p_2,\cdots,p_N)$ is a probability vector if and only if $\sum_{i=1}^{N}p_i=1$ and $p_i>0$ for all $i\in \{1,2,\cdots,N\}$.
	\end{remark}
	\textbf{Note-} Let $\mathcal{I}=\big\{(Y,\rho); \theta_1,\theta_2,\dots,\theta_N \big\}$ be an IFS. We define the Hutchinson mapping $S$ from $H(Y)$ into $H(Y)$ given by
	$$ S(A) = \cup_{i=1}^{N} \theta_i (A).$$
	The map $S$ is a contraction
	map under the Hausdorff metric $h$. If $(Y,\rho)$ is a complete metric space, then, by Banach contraction principle, there exists a unique $E\in H(Y)$ such that $$ E = \cup_{i=1}^N \theta_i (E) .$$ The set $E$ is called the attractor of the IFS.  Furthermore, if $(p_1,p_2, \dots,p_N)$ is a probability vector corresponding to the IFS $\mathcal{I}$, then there exists a unique Borel probability measure $\mu$ supported on the attractor $E$ such that 
	\begin{equation*}
		\begin{aligned}
			\mu &= \sum_{i=1}^{N} p_i \mu \circ \theta_{i}^{-1}.
		\end{aligned}
	\end{equation*}
	We call $\mu$ the invariant measure corresponding to the WIFS $\{(Y,\rho); \theta_1,\theta_2,\dots,\theta_N; p_1,p_2,\cdots,p_N\}$. We refer the reader to see \cite{MF2,Fal}
	for details.
	
	
	\begin{definition}
		We say that an IFS $\mathcal{I}=\{(Y,\rho);\theta_1,\theta_2,\dots,\theta_N\}$ satisfies the open set condition (OSC) if there is a non-empty open set $O$ with $\theta_i(O) \subset O~~\forall~i\in \{1,2,\cdots,N\}$  and  $ \theta_i(O)\cap \theta_j(O)= \emptyset $ for $i\ne j$. Moreover, if $O \cap E \ne \emptyset,$  where $E$ is the attractor of the IFS $\mathcal{I}$, then we say that  $\mathcal{I}$ satisfies the strong open set condition (SOSC). If $\theta_i(E)\cap \theta_j(E)=\emptyset$ for $i\ne j$, then we say that the IFS $\mathcal{I}$ satisfies the strong separation condition (SSC).  
	\end{definition}
	
	\subsection{Code space}
	For this part, we refer the reader to \cite{Fal}.
	Let $(X, \rho)$ be a complete metric space. 
	Let  $\mathcal{F}:=\{X; f_1,f_2, \dots ,f_N\}$ be an IFS.
	Let $ \Omega$ be the set of all infinite sequences $\Omega=\{\{\sigma_k\}_{k=1}^{\infty}; \sigma_k\in \{1,2,\dots ,N\} \}.$  We write $ \sigma= \sigma_1\sigma_2\sigma_3\dots  \in \Omega$ to denote a typical element of $\Omega,$ and we write $\sigma_k$ to denote the $k$th element of $\sigma \in \Omega.$ Then $(\Omega, d_{\Omega})$ is a compact metric space, where the metric $d_{\Omega}$ is defined by $d_{\Omega}(\sigma, \omega)= 0$ when $\sigma= \omega$ and $d_{\Omega}(\sigma, \omega)= 2^{-k}$ when $k$ is the least index for which $\sigma_k \ne \omega_k.$ We call $\Omega$ the code space associated with the IFS $\mathcal{F}.$
	\par
	
	Let $\sigma \in \Omega $ and $x \in X.$ Then, using the contractivity of $\mathcal{F},$ it is not difficult to prove that $$ \phi_{\mathcal{F}} (\sigma) := \lim_{k \rightarrow \infty} f_{\sigma_1} \circ f_{\sigma_2} \circ \dots \circ f_{\sigma_k}(x)$$
	exists, is independent of $x,$ and depends continuously on $\sigma.$ 
	Let $A_{\mathcal{F}}=\{\phi_{\mathcal{F}}(\sigma): \sigma \in \Omega \}.$ Then, it is easy to show that $A_{\mathcal{F}} \subset X$ is the attractor of $\mathcal{F}.$ The continuous function $$ \phi_{\mathcal{F}}: \Omega \rightarrow A_{\mathcal{F}}$$ is called the address function of $\mathcal{F}.$ We call $ \phi_{\mathcal{F}}^{-1}(\{x\})=\{\sigma \in \Omega:\phi_{\mathcal{F}}(\sigma)=x\} $ the set of addresses of the point $x \in A_{\mathcal{F}}.$ \\
	We order the elements of $\Omega$ according to $$\sigma \prec \omega~ \text{if and only if}~ \sigma_k < \omega_k,$$ where $k$ is the least index for which $\sigma_k \ne \omega_k .$ We observe that all elements of $\Omega $ are less than or equal to $ \overline{N}= NNN\dots $ and greater than or equal to $ \overline{1}= 111\dots .$ Note that $\phi_{\mathcal{F}}^{-1} (\{x\})$ contains a unique largest element.
	Let $\mathcal{F}$ be an IFS with attractor $A_{\mathcal{F}}$ and address function $\phi_{\mathcal{F}}: \Omega \rightarrow A_{\mathcal{F}}.$ Let $$ \tau_{\mathcal{F}}(x)=\max\{\sigma \in \Omega : \phi_{\mathcal{F}}(\sigma)=x\}$$
	for all $x \in A_{\mathcal{F}}.$ Then $$ \Omega_{\mathcal{F}}:=\{ \tau_{\mathcal{F}}(x): x \in A_{\mathcal{F}}\} $$ is called the tops code space and $$ \tau_{\mathcal{F}}: A_{\mathcal{F}} \rightarrow \Omega_{\mathcal{F}}$$ is called the tops function corresponding to the IFS $\mathcal{F}.$ It can be seen that the tops function $ \tau_{\mathcal{F}}: A_{\mathcal{F}} \rightarrow \Omega_{\mathcal{F}}$ is one-one and onto.
	\par
	The address structure of the IFS $\mathcal{F}$ is defined to be the set of sets 
	$$\mathcal{C}_\mathcal{F}=\{ \phi_{\mathcal{F}}^{-1}(x)\cap \overline{ \Omega}_{\mathcal{F}} : x\in A_\mathcal{F}  \}.$$
	Let $\mathcal{C}_\mathcal{G}$ be the address structure of another IFS $\mathcal{G}=\{Y; g_1,g_2,\cdots,g_N\}.$  We write $\mathcal{C}_\mathcal{F}\prec \mathcal{C}_\mathcal{G}$, if for each $P\in \mathcal{C}_\mathcal{F}$, there is a $Q\in \mathcal{C}_\mathcal{G}$  such that $P\subset Q$.
	\begin{definition}
		Let $\mathcal{F}=\{X; f_1,f_2,\cdots ,f_N\}$ and $\mathcal{G}=\{Y; g_1,g_2,\cdots,g_N\}$ be two IFSs. Suppose $A_\mathcal{F}$  and $A_\mathcal{G}$ are the attractors of $\mathcal{F}$  and $\mathcal{G},$ respectively. The associated fractal transformation $T_{\mathcal{FG}}: A_\mathcal{F}\to A_\mathcal{G}$ is defined by  $$T_{\mathcal{FG}}= \phi_{\mathcal{G}}\circ \tau_{\mathcal{F}},$$
		where $\phi_{\mathcal{G}}$ is the code map corresponding to $\mathcal{G}$ and $\tau_{\mathcal{F}}$ is the tops function corresponding to  $\mathcal{F}$. 
	\end{definition}
	\begin{remark}\cite[Theorem 1]{MF7}
		Let $\mathcal{F}=\{X; f_1,f_2,\cdots ,f_N\}$ and $\mathcal{G}=\{Y; g_1,g_2,\cdots,g_N\}$ be two IFS such that $\mathcal{C}_\mathcal{F}\prec \mathcal{C}_\mathcal{G}$. Then, the fractal transformation map $T_{\mathcal{FG}}$ is continuous.
	\end{remark}

	\par
	Let $(X,\rho)$ be a complete metric space. Given a Borel probability measure $\mu $ on $X$, a number $r \in (0, +\infty)$ and $ n \in \mathbb{N}$, the $n$th quantization error of order $r$ for $\mu $ is defined by $$V_{n,r}(\mu):=\inf \Big\{\int \rho(x, A)^r d\mu(x): A \subset X, \, \text{Card}(A) \le n\Big\},$$
	where $\rho(x, A)$ represents the distance of the point $x$ from the set $A$. Let $e_{n,r}(\mu)=V_{n,r}^{\frac{1}{r}}(\mu)$. We define the \textit{quantization dimension} of order $r$ of $\mu $ by $$D_r=D_r(\mu):= \lim_{n \to \infty} \frac{ \log n}{- \log \big(e_{n,r}(\mu)\big)},$$
	if the limit exists.
	If the limit does not exist, then we define the {lower} and {upper quantization dimensions} by taking the limit inferior and the limit superior of the sequence and are denoted by $\underline{D}_r$ and $\overline{D}_r$, respectively. 
	\begin{remark}
		Let $\mu$ be a Borel probability measure on $\mathbb{R}^d$ with $ \int \|x\|^r d\mu(x)<\infty$. Then for every $n\in \mathbb{N}$ there exists a finite set $A_n\subset \mathbb{R}^d$ such that
		$$V_{n,r}(\mu)=\int \min_{a\in A_n}\|x-a\|^rd\mu(x).$$
		This $A_n$ is called an $n$-optimal set for measure $\mu$ of order $r$.
	\end{remark}
	\begin{remark}
		Let $\mu$ be a Borel probability measure on $\mathbb{R}^d$ with compact support. Then
		$$\int \|x\|^r d\mu(x)<\infty,$$ for any $r\in (0,+\infty).$
	\end{remark}
	\begin{lemma}\cite{GL1}\label{le2.9}
		Let $\mu$ be a Borel probability measure on $\mathbb{R}^d$ with $ \int \|x\|^r d\mu(x)<\infty$. Then for any $r\in (0,\infty)$
		$$V_{n,r}(\mu)\to 0,~~\text{as}~~n\to \infty.$$
	\end{lemma}
	\begin{definition}
		Let $A$ be a subset of $\mathbb{R}^d.$ The Voronoi region of $a\in A$ is defined by
		$$W(a|A)=\{x\in \mathbb{R}^d:~~\|x-a\|=\min\limits_{b\in A}\|x-b\|\}$$
		and the set $\{W(a|A);~~a\in A\}$ is called the Voronoi diagram of $A$. 
	\end{definition}
	\begin{lemma}\cite{GL1}\label{le2.11}
		Let $\mu$ be a Borel probability measure on $\mathbb{R}^d$ with compact support $A_{*}$ and $r\in (0,\infty)$. Let $A_n$ be an $n$-optimal set for measure $\mu$ of order $r$. Define
		$$\|A_n\|_{\infty}=\max\limits_{a\in A_n}~~\max\limits_{x\in W(a|A_n)\cap A_{*}}\|x-a\|.$$ 
		Then $$\bigg(\frac{\|A_n\|_{\infty}}{2}\bigg)^r\min\limits_{x\in A_{*}}\mu\bigg(B\bigg(x,\frac{\|A_n\|_{\infty}}{2}\bigg)\bigg)\leq V_{n,r}(\mu).$$
	\end{lemma}
	\begin{proposition}\cite{GL1}\label{prop2.12}
		\begin{enumerate}
			\item If $0\leq t_1<\overline{D}_r<t_2,$ then
			$$\limsup_{n\to \infty}n\cdot e_{n,r}^{t_1}=\infty~~~\text{and}~~~\lim_{n\to \infty}n\cdot e_{n,r}^{t_2}=0.$$
			\item If $0\leq t_1<\underline{D}_r<t_2,$ then
			$$\liminf_{n\to \infty}n\cdot e_{n,r}^{t_2}=0~~~\text{and}~~~\lim_{n\to \infty}n\cdot e_{n,r}^{t_1}=\infty.$$
		\end{enumerate}
		
	\end{proposition}
	
	\par Let $\{(X,\rho);f_1,f_2,\cdots,f_N\}$ be an IFS with probability vector $(p_1,p_2,\cdots,p_N )$. We denote the set of all finite sequences of symbols belonging to the set $\{1,2,\cdots, N\}$ by $\{1,2,\cdots,N\}^*$ and $|\sigma|$ denotes the length of $\sigma \in \{1,2,\cdots,N\}^*.$ We denote the set of all finite sequences of length $n$ over the symbols belonging to the set  $\{1,2,\cdots, N\}$ by $\{1,2,\cdots,N\}^n$. Let $ \sigma \in  \{1,2,\cdots,N\}^*$ and $m\leq |\sigma|$, we define $\sigma|_{m}$ as follows 
	\begin{equation*}
		\sigma|_{m}=
		\begin{cases}
			\emptyset  ~~~~,\quad\quad m=0\\
			\sigma_1\sigma_2\cdots\sigma_m, m\ne0.
		\end{cases}
	\end{equation*}
	We define a natural order on $\{1,2,\cdots,N\}^*$ by 
	$$\sigma\leq \tau~~~~~\text{iff}~~~~~|\sigma|\leq |\tau|, \tau_{|\sigma|}=\sigma,$$ where $\sigma,\tau\in \{1,2,\cdots,N\}^*$. We write $p_{\sigma} = p_{\sigma_1}\cdot p_{\sigma_2}\cdots p_{\sigma_m}$ and
	$p_{\sigma^-} = p_{\sigma_1}\cdot p_{\sigma_2}\cdots p_{\sigma_{m-1}}$ for $\sigma\in \{1,2,\cdots,N\}^*$, $|\sigma|=m$.  Let $\sigma,\tau\in \{1,2,\cdots,N\}^*$. We say that $\sigma$ and $\tau$ are incomparable if neither $\sigma\leq \tau$ nor $\tau\leq \sigma.$ 
	\par
	A finite set $\Gamma\subset\{1,2,\cdots,N\}^*$  is called a finite antichain if and only if any two elements of $\Gamma$ are incomparable. A finite antichain $\Gamma$ is called maximal if and only if for every finite antichain $\Gamma'\subset \{1,2,\cdots,N\}^* $ with $\Gamma\subseteq \Gamma'$, we have $\Gamma' = \Gamma.$
	\par
	\begin{lemma}\label{le3.15}\cite{GL1}\cite[Lemma 3.8]{R2}
		Let $\{\mathbb{R}^{d};f_1,f_2,\cdots,f_N; p_1,p_2,\cdots,p_N\}$ be a WIFS. Let $\mu$ be the invariant measure corresponding to this WIFS. If $\Gamma $ is a finite maximal  antichain, then 
		$$\sum_{\sigma\in \Gamma}p_\sigma=1 ~~~\text{and}~~~ \mu=\sum_{\sigma\in \Gamma }p_\sigma\mu \circ f_\sigma^{-1}.$$
	\end{lemma}
	
	\begin{lemma}\label{le3.16}\cite{GL1}
		Let $\{\mathbb{R}^{d};f_1,f_2,\cdots,f_N;p_1,p_2,\cdots,p_N\}$ be a WIFS. If  $~~ 0<\epsilon\leq \min\{p_1,p_2,\\ \cdots,p_N\}$. Then
		$$\Gamma_{\epsilon}=\{\sigma\in \{1,2,\cdots,N\}^*;~~ p_{\sigma^-}\geq \epsilon >p_{\sigma}\}$$ is a finite maximal antichain.
	\end{lemma}
	Let $\{\mathbb{R}^d; f_1,f_2, \dots, f_N;p_1,p_2, \dots, p_N\}$ be a WIFS such that each $f_i $ is a contractive similarity transformation such that
	$$\|f_i(x_1)-f_i(x_2)\|=c_i\|x_1-x_2\|,$$ 
	where $0<c_i<1.$
	Then, there is a unique Borel probability measure $\mu $ supported on the attractor $E$ such that  $$ \mu = \sum_{i=1}^{N} p_i  \mu \circ f_i^{-1}.$$
	In this case, we call the measure $\mu $ an invariant self-similar measure. 
	Graf and Luschgy \cite{GL1,GL2} proved that the quantization dimension function $l_r$ of the invariant self-similar measure $\mu $ exists, and satisfies the following equation: $$\sum_{i=1}^N (p_i c_i^r)^{\frac{l_r}{r +l_r}}=1,$$
	provided that the given WIFS satisfies the OSC.
	\par

	\section{On the product IFS and dimension of the graph of a fractal transformation}\label{se3}

	In the following theorem, we determine the bounds of the Hausdorff dimension of graph of fractal transformation without any separation condition.
	\begin{theorem}\label{boundth}
		Let $\mathcal{F}=\{(X,\rho_1); f_1,f_2,\cdots ,f_N\}$ and $\mathcal{G}=\{(Y,\rho_2); g_1,g_2,\cdots,g_N\}$ be two IFSs. Consider the IFS $\mathcal{H}=\{X \times Y; h_1,h_2,\cdots ,h_N\}$, where $h_i(x,y)=(f_i(x),g_i(y))$. Then  
		$$\dim_HA_\mathcal{F}\leq \dim_H G(T_\mathcal{FG})\leq \dim_HA_\mathcal{H},$$ where $G(T_\mathcal{FG})$ denotes the graph of the fractal transformation $T_\mathcal{FG}$.
	\end{theorem}
	\begin{proof}
		We define a mapping $\Psi: G(T_{\mathcal{FG}})\to A_\mathcal{F}$ by
		$$\Psi(x,T_\mathcal{FG}(x))=x.$$
		Let $(x,T_\mathcal{FG}(x)),(x',T_\mathcal{FG}(x'))\in G(T_{\mathcal{FG}}) $. Then, we have 
		\begin{align*}
			\rho_1(\Psi(x,T_\mathcal{FG}(x)),\Psi(x',T_\mathcal{FG}(x')))&=\rho_1(x,x')\\
			&\leq \max \{\rho_1(x,x'),\rho_2(T_\mathcal{FG}(x),T_\mathcal{FG}(x'))\}\\
			&= \mathcal{D}((x,T_\mathcal{FG}(x)),(x',T_\mathcal{FG}(x'))).
		\end{align*}
		Thus, $\Psi$ is a Lipschitz map. Therefore, by the Lipschitz invariance property of the Hausdorff dimension, we get
		\begin{equation}\label{e3.1}
			\dim_HA_\mathcal{F}\leq \dim_H G(T_\mathcal{FG}).
		\end{equation}
		For the other inequality, let $(x,T_\mathcal{FG}(x))\in G(T_\mathcal{FG}). $  So, there is a $\sigma\in \Omega_\mathcal{F}$  such that $\phi_\mathcal{F}(\sigma)=x.$ Now,
		$$(x,T_\mathcal{FG}(x))=(\phi_\mathcal{F}(\sigma),\phi_{\mathcal{G}}\circ \tau_{\mathcal{F}}\circ\ \phi_\mathcal{F}(\sigma))=(\phi_\mathcal{F}(\sigma),\phi_\mathcal{G}(\sigma))\in A_\mathcal{H}.$$
		Hence, $G(T_\mathcal{FG})\subset{A_\mathcal{H}}.$ Therefore, using the monotonic property of the Hausdorff dimension, we have 
		\begin{equation}\label{e3.2}
			\dim_H G(T_\mathcal{FG})\leq \dim_HA_\mathcal{H}. 
		\end{equation}
		Combining inequalities \ref{e3.1} and \ref{e3.2}, we get our required  result.    
	\end{proof}
	
	In the next lemma, we show that the product IFSs satisfies SSC, OSC and SOSC provided one of the IFSs satisfies SSC, OSC and SOSC, respectively. 
	
	The proof of the following lemma is not difficult. But, we give the complete proof for the convenience of the reader.   
	\begin{lemma}\label{Cond}
		Let $\mathcal{F}=\{X; f_1,f_2,\cdots ,f_N\}$ and $\mathcal{G}=\{Y; g_1,g_2,\cdots,g_N\}$ be two IFSs. For the IFS $\mathcal{H}:=\{X \times Y; h_1,h_2,\cdots ,h_N\}$, where $h_i(x,y)=(f_i(x),g_i(y))$, we have the following
		\begin{enumerate}
			\item If $\mathcal{F}$ satisfies SSC, then $\mathcal{H}$ also satisfies SSC.
			\item If $\mathcal{F}$ satisfies OSC, then $\mathcal{H}$ also satisfies OSC.
			\item If $\mathcal{F}$ satisfies SOSC,  then $\mathcal{H}$ also satisfies SOSC.
		\end{enumerate}
	\end{lemma}
	\begin{proof}
		\begin{enumerate}
			\item Since $\mathcal{F}$ satisfies SSC, we have 
			\[
			f_i(A_\mathcal{F}) \cap f_j(A_\mathcal{F})= \emptyset  ~~ \text{for all}~ i \ne j.\]
			This further yields
			\[
			h_i(A_\mathcal{F}\times A_\mathcal{G}) \cap h_j(A_\mathcal{F}\times A_\mathcal{G})=(f_i(A_\mathcal{F})\times g_i(A_\mathcal{G})) \cap (f_j(A_\mathcal{F})\times g_j(A_\mathcal{G}))= \emptyset ,~ \text{for all}~ i \ne j.\]
			Since $A_{\mathcal{H}} \subset A_\mathcal{F} \times A_\mathcal{G},$ we obtain
			\[
			h_i(A_{\mathcal{H}}) \cap h_j(A_{\mathcal{H}})= \emptyset ~ \text{for all}~ i \ne j.\]
			Therefore, the IFS  $\mathcal{H}$ satisfies SSC.
			\item Since $\mathcal{F}$ satisfies  OSC, we have an open set $U$ such that
			\[\cup_{i=1}^Nf_i(U) \subset U~~\text{and}~~
			f_i(U) \cap f_j(U)= \emptyset ~~ \text{for all}~ i \ne j.\]
			This further yields
			\[
			h_i(U\times Y) \cap h_j(U \times Y)=(f_i(U)\times g_i(Y)) \cap (f_j(U)\times g_j(Y))= \emptyset ~ \text{for all}~ i \ne j.\]
			Now, define an open set $W = U \times Y.$ Then
			\[
			\cup_{i=1}^N h_i(W) \subset W ~~ \text{and}~~h_i(W) \cap h_j(W)= \emptyset ~ \text{for all}~ i \ne j.\]
			Therefore, the IFS  $\mathcal{H}$ satisfies OSC.
			\item Since $\mathcal{F}$  satisfy SOSC, we have an open set $U$ such that
			\[\cup_{i=1}^Nf_i(U) \subset U,~~U\cap A_\mathcal{F}\ne\emptyset,
			f_i(U) \cap f_j(U)= \emptyset ~~ \text{for all}~ i \ne j.\]
			This further yields
			$$h_i(U\times Y)= (f_i(U)\times g_i(Y))\subset{U\times Y},~~(U\times Y) \cap (G(T_{\mathcal{FG}})) \ne \emptyset,$$
			$$h_i(U\times Y)\cap h_j(U\times Y)=(f_i(U)\times g_i(Y))\cap (f_j(U)\times g_j(Y))=\emptyset~~\text{for all}~~i\ne j.$$
			Now, we define an open set $W = U \times Y.$ Since $G(T_{\mathcal{FG}}) \subset{A_\mathcal{H}} $, we have
			\[
			\cup_{i=1}^N h_i(W) \subset W,~~h_i(W) \cap h_j(W)= \emptyset ~ \text{for all}~ i \ne j~~\text{and}~W\cap (A_{\mathcal{H}}) \ne \emptyset.\]
			Therefore, the IFS  $\mathcal{H}$ satisfies SOSC.

		\end{enumerate}
		
	\end{proof}
	Next, we define a metric on the product space as follows:
	\par Let $(X,\rho_1)$ and $(Y,\rho_2)$ be two complete metric spaces.  We define a metric $\mathcal{D}$ on $X\times Y$ by
	$$\mathcal{D}((x,y),(x',y'))=\max\{\rho_1(x,x'),\rho_2(y,y')\},$$
	where $x,x'\in X$ and $y,y'\in Y.$ It is well-known that $(X\times Y,\mathcal{D})$ is a complete metric space.
	\begin{lemma}\label{sibi}
		Let $\mathcal{F}=\{(X,\rho_1); f_1,f_2,\cdots ,f_N\}$ and $\mathcal{G}=\{(Y,\rho_2); g_1,g_2,\cdots,g_N\}$ be two IFSs. For the IFS $\mathcal{H}:=\{X \times Y; h_1,h_2,\cdots ,h_N\}$, where $h_i(x,y)=(f_i(x),g_i(y))$, we have the following
		\begin{enumerate}
			\item If $f_i$'s and $g_i$'s are similarity transformations such that $\rho_1(f_i(x),f_i(x'))=c_i\rho_1(x,x') ~~\text{and}~~\\ \rho_2(g_i(y),g_i(y'))=c_i\rho_2(y,y')$ for all $x,x'\in X$ and $y,y'\in Y$, then so are $h_i$'s.
			\item If $f_i$'s and $g_i$'s are bi-Lipschitz mappings then so are $h_i$'s.
		\end{enumerate}
	\end{lemma}
	\begin{proof} Since the proof is easy, we skip it.
		
	\end{proof}
	\begin{lemma}
		Let $\mathcal{F}=\{\mathbb{R}^d; f_1,f_2,\cdots ,f_N\}$ and  $\mathcal{G}=\{\mathbb{R}^d; g_1,g_2,\cdots,g_N\}$ be two IFSs. Consider the IFS $\mathcal{H}:=\{\mathbb{R}^d \times \mathbb{R}^d; h_1,h_2,\cdots ,h_N\}$, where $h_i(x,y)=(f_i(x),g_i(y))$. If $f_i$'s and $g_i$'s are affine transformations, then so are $h_i$'s.
	\end{lemma}
	\begin{proof}
Since the proof is easy, we skip it.
	\end{proof} 
	\begin{remark}
		We define three mappings $f_1,f_2,f_3: \mathbb{R}\to \mathbb{R}$ by
		$$f_1(x)=\frac{x}{3},~~~~f_2(x)=\frac{x}{2},~~~~ f_3(x)=\frac{x}{2}+\frac{1}{2}.$$
		Let $\mathcal{I}=\{\mathbb{R}; f_1,f_2,f_3\}$ and $\mathcal{J}=\{\mathbb{R}; f_2,f_3\}$ be two IFSs. Then $\mathcal{J}$ is a sub IFS of $\mathcal{I}.$  One can easily show that the sub IFS $\mathcal{J}$ satisfies SOSC and SOC but the IFS $\mathcal{I}$ does not.    
	\end{remark}

	In the following proposition, we give the exact value of the Hausdorff dimension of the graph of the fractal transformation under some separation condition. 
	\begin{proposition}\label{p3.9}
		Let $\mathcal{F}=\{(X,\rho_1); f_1,f_2,\cdots ,f_N\}$ and $\mathcal{G}=\{(Y,\rho_2); g_1,g_2,\cdots,g_N\}$ be two IFSs  such that
		$$\rho_1(f_i(x),f_i(x'))=c_i\rho_1(x,x'),$$
		$$\rho_2(g_i(y),g_i(y'))\leq r_i\rho_2(y,y'),$$
		where $c_i,r_i\in (0,1).$ If $r_i\leq c_i$ for all $i\in \{1,2,\cdots,N\}$, $\mathcal{F}$ satisfies SOSC, then 
		$\dim_HG(T_\mathcal{FG})=s_0$, where $s_0$ is given by $\sum_{i=1}^{N}{c_i}^{s_0}=1.$ 
	\end{proposition}
	\begin{proof}
		If we consider the IFS $\mathcal{H}=\{X \times Y; h_1,h_2,\cdots ,h_N\}$. Then
		\begin{align*}
			\mathcal{D}(h_i(x,y),h_i(x',y'))&=\mathcal{D}((f_i(x),g_i(y)),(f_i(x'),g_i(y')))\\
			&=\max(\rho_1(f_i(x),f_i(x')),\rho_2(g_i(y),g_i(y')))\\
			&\leq \max (c_i\rho_1(x,x'),r_i\rho_2(y,y'))\\
			&\leq c_i \mathcal{D}((x,y),(x',y')).
		\end{align*}
		Therefore  by  \cite[Theorem 9.6]{Fal}, $\dim_HA_\mathcal{H}\leq s_0$ where $s_0$ is given by $\sum\limits_{i=1}^{N}{c_i}^{s_0}=1$.
		Since $\mathcal{F}$ satisfies SOSC and each $f_i$ is a similarity transformation, by \cite[Theorem 2.6]{AS} $\dim_HA_\mathcal{F}= s_0$, where $s_0$ is uniquely determined by $\sum\limits_{i=1}^{N}{c_i}^{s_0}=1$. Theorem \ref{boundth}, yields that
		$$s_0=\dim_HA_\mathcal{F}\leq \dim_HG(T_\mathcal{FG})\leq \dim_HA_\mathcal{H}\leq s_0.$$
		Hence  $\dim_HG(T_\mathcal{FG})=s_0$. This completes the proof.  
	\end{proof}
	\begin{proposition}
		Let $\mathcal{F}=\{(X,\rho_1); f_1,f_2,\cdots ,f_N\}$ and $\mathcal{G}=\{(Y,\rho_2); g_1,g_2,\cdots,g_N\}$ be two IFSs  such that
		$$\rho_1(f_i(x),f_i(x'))=c_i\rho_1(x,x'),$$
		$$\rho_2(g_i(y),g_i(y'))\leq r_i\rho_2(y,y'),$$
		where $c_i,r_i\in (0,1).$ If $c_i\leq r_i$ for all $i\in \{1,2,\cdots,N\}$, $\mathcal{F}$ satisfies  SOSC. Then 
		$s_0 \leq \dim_HG(T_\mathcal{FG})\leq t_0$, where $t_0$ and  $s_0$ are given by  $\sum_{i=1}^{N}{r_i}^{t_0}=1$ and  $\sum_{i=1}^{N}{c_i}^{s_0}=1$ respectively. 
	\end{proposition}
	\begin{proof}
		Using similar arguments of Proposition \ref{p3.9}, one can easily prove this.  
	\end{proof}

	In the upcoming theorem, we obtain bounds of the Hausdorff dimension of the product IFS provided some conditions hold.
	\begin{theorem}\label{bidi}
		Let IFS  $~\mathcal{H}:=\{X \times Y; h_1,h_2,\cdots ,h_N\}$ satisfy SOSC , where $h_i(x,y)=(f_i(x),g_i(y))$, and assume that $$  c_i \mathcal{D}((x,y),(x',y') ) \le \mathcal{D}(h_i(x,y) , h_i(x',y')) \le C_i \mathcal{D}((x,y),(x',y')) ,$$ where  $(x,y),(x',y') \in X \times Y$ and $0 < c_i \le C_i < 1 ~ \forall~ i \in \{1,2,\cdots,N\} .$ Then $r \le  \dim_H(A_\mathcal{H}) \le R  ,$ where $r$ and $R$ are given by  $ \sum\limits_{i=1}^{N} c_i^{r} =1$ and $ \sum\limits_{i=1}^{N} C_i^{R} =1$, respectively.
	\end{theorem}
	
	\begin{proof}
		For the upper bound of $\dim_H(A_\mathcal{H})$, follow Proposition $9.6$ in \cite{Fal}. For the lower bound of $\dim_H(A_\mathcal{H})$ we proceed as follows.
		
		Since the IFS $\mathcal{H}$ satisfies SOSC, there exists an open set $V$  of $X\times Y$ such that    \[\cup_{i=1}^N h_i(V) \subset V,~~V\cap A_\mathcal{H}\ne\emptyset,
		h_i(V) \cap h_j(V)= \emptyset ~~ \forall~ i \ne j,~~1\leq i,j\leq N.\]
		Since $V\cap A_{\mathcal{H}}\ne \emptyset$, therefore there exists an index $\omega \in \{1,2,\cdots,N\}^*$ such that $h_\omega(A_\mathcal{H})\subset V.$ We denote  $ h_\omega(A_\mathcal{H})$ by $ (A_\mathcal{H})_\omega$ for any $\omega\in \{1,2,\cdots,N\}^*$. Now, by using the condition $h_i(V) \cap h_j(V)= \emptyset ~~ \forall~ i \ne j,~~1\leq i,j\leq N$ and $h_\omega(A_\mathcal{H})\subset V,$ it is clear that for each $n\in \mathbb{N}$, the sets $\{(A_\mathcal{H})_{i\omega}: i \in \{1,2,\cdots,N\}^{n} \}$ are  pairwise disjoint. We define an IFS $\mathcal{L}_n=\{h_{i\omega}: i \in \{1,2,\cdots,N\}^n\}$. Let $A_n^*$ be the attractor of IFS $\mathcal{L}_n$. By analysing the code space of IFS $\mathcal{L}_n$ and $\mathcal{H}$, we deduce that $A_n^*\subset A_{\mathcal{H}}$. This further yields that the IFS $\mathcal{L}_n=\{h_{i\omega}: i \in \{1,2,\cdots,N\}^n\}$ satisfies SSC. Thus, the IFS $\mathcal{L}_n$ fulfill all the assumptions of Proposition $9.7$ in \cite{Fal}. Hence, by Proposition $9.7$ in \cite{Fal}, we obtain that $ r_n \le \dim_H(A_n^*)$, where $r_n$ is given by  $ \sum_{ i \in \{1,2,\cdots,N\}^n} c_{i\omega}^{r_n} =1.$ Then  $  r_n \le \dim_H(A_n^*) \le \dim_H(A_\mathcal{H})$ because $A_n^*\subset A_\mathcal{H}$. Suppose that $ \dim_H(A_\mathcal{H}) < r.$ This implies that  $ r_n < r $. Let $ c_{max}=\max\{c_1, c_2, \dots,c_{N}\}.$ Then, we have
		$$
		c_{\omega}^{- r_n}  = \sum_{ i \in \{1,2,\cdots,N\}^n} c_{i}^{r_n}\ \ge \sum_{ i \in \{1,2,\cdots,N\}^n} c_{i}^{r} c_{i}^{\dim_H(A_\mathcal{H}) -r} \ge \sum_{ i \in \{1,2,\cdots,N\}^n} c_{i}^{r} c_{max}^{n(\dim_H(A_\mathcal{H}) - r)} $$
		This implies that $$c_{\omega}^{- r} \geq c_{max}^{n(\dim_H(A_\mathcal{H}) -r)}. $$ 
		We have a contradiction for large values of $n\in \mathbb{N} $. Therefore, we get $ \dim_H(A_\mathcal{H}) \ge  r,$ proving the assertion.
	\end{proof} 
\begin{remark}
  We may compare the above result with the work of Edgar and Golds \cite{Edgar}, wherein authors determined the upper and the lower  bounds of the  Hausdorff dimension of the graph directed bi-Lipschitz IFS using the Perron-Frobenius theory. We have obtained the upper and the lower bounds  of the Hausdorff dimension for the bi-Lipschitz IFS using Moran-Hutchinson technique. We note that our result can be obtained as a special case of Edgar and Golds \cite{Edgar}.	However, our technique is different and the proof is simpler. For determining the Hausdorff dimension using the Perron-Frobenius theory in more general setting,  one can see \cite{Nussbaum1}.  
\end{remark}

	In the next result, we estimate bounds of the Hausdorff dimension of the graph of the fractal transformation by using the previous theorem.
	\begin{proposition}\label{bigr}
		Let $\mathcal{F}=\{(X, \rho_1); f_1,f_2,\cdots ,f_N\}$ and $\mathcal{G}=\{(Y, \rho_2); g_1,g_2,\cdots,g_N\}$ be two IFSs such that $f_i$ and $g_i$ are bi-Lipschitz mappings as follows
		$$c_i \rho_1(x,x')\leq \rho_1(f_i(x),f_i(x'))\leq c_i' \rho_1(x,x'),$$
		$$r_i \rho_2(y,y')\leq \rho_2(g_i(y),g_i(y'))\leq r_i' \rho_2(y,y'),$$
		where $c_i,c_i',r_i,r_i'\in \mathbb{R}$, $0<c_i\leq c_i'<1$ and $0<r_i\leq r_i'<1.$ Assume that $\mathcal{F}$ satisfies SOSC and $\mathcal{C}_\mathcal{F}\prec \mathcal{C}_\mathcal{G}.$ Then
		$s_1\leq \dim_H{G(T_{\mathcal{FG}})}\leq s_2$, where $s_1$ and $s_2$ are uniquely determined by $\sum\limits_{i=1}^{N}\min\{c_i,r_i\}^{s_1}=1$  and $\sum\limits_{i=1}^{N}\max\{c_i',r_i'\}^{s_2}=1$, respectively.
	\end{proposition}
	\begin{proof}
		Since $\mathcal{F}$ satisfies SOSC and $\mathcal{C}_\mathcal{F}\prec \mathcal{C}_\mathcal{G}$, by \cite[Theorem 6.8]{And}, the graph of the fractal transformation $T_\mathcal{FG}$ is same as the attractor of the IFS $\mathcal{H}=\{X\times Y; h_1,h_2,\cdots,h_N\}$, where $h_i(x,y)=(f_i(x),g_i(y))$ for all $i\in \{1,2,\cdots,N\}$. In the light of  part(2) of Lemma \ref{sibi}, part(3) of Lemma \ref{Cond} and Theorem \ref{bidi}, we get our required result.
		
	\end{proof}
	\begin{corollary}
		Let $\mathcal{F}=\{(X,\rho_1); f_1,f_2,\cdots ,f_N\}$ and $\mathcal{G}=\{(Y,\rho_2); g_1,g_2,\cdots,g_N\}$ be two IFSs such that $f_i$ and $g_i$ are similarity transformation as follows $$\rho_1(f_i(x),f_i(x'))=c_i\rho_1(x,x'),~~~~\rho_2(g_i(y),g_i(y'))=c_i\rho_2(y,y'),$$ where $c_i\in (0,1).$ If $\mathcal{F}$ satisfies SOSC and $\mathcal{C}_\mathcal{F}\prec \mathcal{C}_\mathcal{G},$ then $\dim_H{G(T_\mathcal{FG})}=s_0$, where $s_0$ is uniquely determined by $\sum\limits_{i=1}^{N}{c_i}^{s_0}=1.$
	\end{corollary}
	\begin{proof}
		This is a direct consequence of Proposition \ref{bigr}.
	\end{proof}
	Next, we discuss some aspects of the invariant measure corresponding to an IFS. 
	\par Consider an IFS $\mathcal{I}=\{[0,1]; f_1,f_2\}$, where $f_1$ and $f_2$ are defined as follows
	$$f_1(x)=\frac{x}{2},~~~ f_2(x)=\frac{x}{2}+\frac{1}{2}.$$
	Then, $[0,1]$ is the attractor of IFS $\mathcal{I}.$ If we take two probability vectors $p=(\frac{1}{2} ,\frac{1}{2})$ and $q=(\frac{1}{3},\frac{2}{3}),$ then the invariant measures $\mu_p$ and $\mu_q$ corresponding to IFS $\mathcal{I}$ with probability vectors $p$  and $q$, respectively, satisfy 
	$$\mu_p=\frac{1}{2}\mu_p\circ f_1^{-1}+\frac{1}{2}\mu_p\circ f_2^{-1},~~~\mu_p[0,1]=1,$$
	$$\mu_q=\frac{1}{3}\mu_q\circ f_1^{-1}+\frac{2}{3}\mu_q\circ f_2^{-1},~~~\mu_q[0,1]=1.$$ 
	Since IFS $\mathcal{I}$  satisfies the OSC and $\partial{[0,1]}=\{0,1\}$, by using \cite[Theorem 2.2]{Bandt}, we get $\mu_p(\{0\})=\mu_p(\{1\})=\mu_p(\{0,1\})=\mu_q(\{0\})=\mu_q(\{1\})=\mu_q(\{0,1\})=0.$
	Therefore, $$\mu_p[0,\frac{1}{2}]=\frac{1}{2}\mu_p\circ f_1^{-1}[0,\frac{1}{2}]+\frac{1}{2}\mu_p\circ f_2^{-1}[0,\frac{1}{2}]=\frac{1}{2}\mu_p[0,1]+\frac{1}{2}\mu_p(\{0\})=\frac{1}{2}=\mathcal{L}_{[0,1]}^{1}[0,\frac{1}{2}]. $$
	In fact, it is not difficult to show that $\mu_p=\mathcal{L}_{[0,1]}^{1},$ where $\mathcal{L}_{[0,1]}^{1}$ is the normalized Lebesgue measure on $[0,1].$ 
	And, $$\mu_q[0,\frac{1}{2}]=\frac{1}{3}\mu_q\circ f_1^{-1}[0,\frac{1}{2}]+\frac{2}{3}\mu_q\circ f_2^{-1}[0,\frac{1}{2}]=\frac{1}{3}\mu_q[0,1]+\frac{2}{3}\mu_q(\{0\})=\frac{1}{3}\ne\mathcal{L}_{[0,1]}^{1}[0,\frac{1}{2}]. $$
	So, for different probability vectors, we get different invariant Borel probability measures supported on the attractor of the IFS.  
	
	In the following theorem, we determine a relation between the invariant measures of the IFS $\mathcal{F}$, $\mathcal{G}$ and  $\mathcal{F}\times \mathcal{G}.$
	\begin{theorem}
		Let $\mathcal{F}=\{(X,\rho_1); f_1,f_2,\cdots ,f_N\}$ and  $\mathcal{G}=\{(Y,\rho_2); g_1,g_2,\cdots,g_N\}$ be two IFSs with a probability vector $(p_1,p_2,\dots,p_N)$. Define the WIFS $\mathcal{F} \times \mathcal{G}:=\{X \times Y; \Psi_{ij}; p_ip_j: 1\leq i,j\leq N\}$, where $\Psi_{ij}(x,y)=(f_i(x),g_j(y))$. We denote the invariant measures by $\mu_{\mathcal{F}},\mu_{\mathcal{G}}$ and $\mu_{\mathcal{F}\times \mathcal{G}}$ associated with IFSs $\mathcal{F},\mathcal{G}$ and $\mathcal{F} \times \mathcal{G}$, respectively. We have the following
		$$\mu_{\mathcal{F}\times \mathcal{G}}= \mu_{\mathcal{F}} \times \mu_{\mathcal{G}},$$
		where $(\mu_{\mathcal{F}} \times \mu_{\mathcal{G}})(A \times B)= \mu_{\mathcal{F}}(A) \mu_{\mathcal{G}}(B).$
	\end{theorem}
	\begin{proof}
		With $p_{ij}=p_ip_j$, we have
		$$    \mu_{\mathcal{F}} = \sum_{i=1}^{N} p_i \mu_{\mathcal{F}} \circ f_{i}^{-1},~~\mu_{\mathcal{G}} = \sum_{i=1}^{N} p_i \mu_{\mathcal{G}} \circ g_{i}^{-1} ~\text{and}~\mu_{*} = \sum_{i=1,j=1}^{N} p_{ij} \mu_{*} \circ \Psi_{ij}^{-1}.$$
		Now, 
		\begin{equation*}
			\begin{aligned}
				\mu_{\mathcal{F}} \times \mu_{\mathcal{G}} & = \sum_{i=1}^{N} p_i \mu_{\mathcal{F}} \circ f_{i}^{-1} \times \sum_{j=1}^{N} p_j \mu_{\mathcal{G}} \circ g_{j}^{-1}\\ & =\sum_{i=1,j=1}^{N} p_{ij} \mu_{\mathcal{F}} \circ f_{i}^{-1} \times \mu_{\mathcal{G}} \circ g_{j}^{-1}\\ & =\sum_{i=1,j=1}^{N} p_{ij} (\mu_{\mathcal{F}}  \times \mu_{\mathcal{G}}) ( f_{i}^{-1}, g_{j}^{-1})\\ & =\sum_{i=1,j=1}^{N} p_{ij} (\mu_{\mathcal{F}}  \times \mu_{\mathcal{G}}) \circ \Psi_{ij}^{-1}.
			\end{aligned}
		\end{equation*}
		Since $\mu_{\mathcal{F}\times \mathcal{G}}$ is the unique measure satisfying the equation $\mu_{\mathcal{F}\times \mathcal{G}} = \sum\limits_{i=1,j=1}^{N} p_{ij} \mu_{\mathcal{F}\times \mathcal{G}} \circ \Psi_{ij}^{-1},$ the above equation yields that $\mu_{\mathcal{F}\times \mathcal{G}}= \mu_{\mathcal{F}}  \times \mu_{\mathcal{G}}.$ Thus, the proof of the theorem is complete.  
	\end{proof}
	
	In the next result, we obtain a relationship between the quantization dimension of the invariant measures of the IFSs $\mathcal{F}$, $\mathcal{G}$ and $\mathcal{F}\times \mathcal{G}.$

	\begin{theorem}
		Let $\mathcal{F}=\{\mathbb{R}^d~; f_1,f_2,\cdots ,f_N\}$ and $\mathcal{G}=\{\mathbb{R}^d~; g_1,g_2,\cdots,g_N\}$ be two IFSs such that
		$$\|f_i(x)-f_i(x')\|=c~\|x-x'\|$$ and
		$$\|g_i(y)-g_i(y')\|=c~\|y-y'\|$$ for all $i\in \{1,2,\cdots, N\}$, where $x,x'\in X$, $y,y'\in Y$ and $0<c<1$. Also, assume that  
		$(p_1,p_2,\dots,p_N)$, and $(q_1,q_2,\dots,q_N)$  are the probability vectors corresponding to IFS $\mathcal{F}$ and $\mathcal{G}$, respectively.  Define the WIFS $\mathcal{F} \times \mathcal{G}:=\{ \mathbb{R}^d\times \mathbb{R}^d ~; \Psi_{ij}; p_iq_j: 1\leq i,j\leq N \}$, where $\Psi_{ij}(x,y)=(f_i(x),g_j(y))$.  We denote the invariant measures by $\mu_{\mathcal{F}},\mu_{\mathcal{G}}$ and $\mu_{\mathcal{F}\times \mathcal{G}}$ associated with IFSs $\mathcal{F},\mathcal{G}$ and $\mathcal{F} \times \mathcal{G}$, respectively. If $\mathcal{F}$ and $\mathcal{G}$ satisfy OSC, then  $$D_r > \max \{{D_r}^*,D_r'\},$$
		where ${D_r}^*,D_r'$ and $D_r$ denotes the quantization dimension of order $0<r<\infty$ of $\mu_\mathcal{F},\mu_\mathcal{G}$ and $\mu_{\mathcal{F}\times \mathcal{G}}$, respectively.
	\end{theorem}
	\begin{proof}
		Since the IFSs $\mathcal{F}$ and $\mathcal{G}$ satisfy OSC and each $f_i$'s  and $g_i$'s are similarity  maps with similarity constant $c$, ${D_r}^*$ and $D_r'$ are uniquely determined by 
		\begin{equation}\label{e3.3}
			\sum_{i=1}^N (p_i c^r)^{\frac{{D_r}^*}{r +{D_r}^*}}=1
		\end{equation} and
		\begin{equation}
			\sum_{i=1}^N (q_i c^r)^{\frac{D_r'}{r +D_r'}}=1.
		\end{equation}
		Since each $f_i$'s  and $g_i$'s are similarity  maps with similarity constant $c$, we have 
		\begin{align*}
			\|\Psi_{ij}(x,y)-\Psi_{ij}(x',y')\|&=\|(f_i(x),g_j(y))-(f_i(x'),g_j(y'))\|\\
			&= \sqrt{\|f_i(x)-f_i(x')\|^2+\|g_i(y)-g_i(y')\|^2}\\
			&= c~ \sqrt{\|x-x'\|^2+\|y-y'\|^2}\\
			&= c\|(x,y)-(x',y')\|.
		\end{align*}
		Thus $\Psi_{ij}$ is a similarity  map with similarity constant $c$.
		Since the IFSs $\mathcal{F}$ and $\mathcal{G}$ satisfy OSC, there exist open sets $U$ and $V$ such that
		$$\cup_{i=1}^Nf_i(U) \subset U,~~
		f_i(U) \cap f_j(U)= \emptyset ~~ \text{for all}~ i \ne j, \text{and}$$
		$$\cup_{i=1}^Ng_i(V) \subset V,~~
		g_i(V) \cap g_j(V)= \emptyset ~~ \text{for all}~ i \ne j.$$
		This further yields,
		$$\Psi_{ij}(U\times V)=f_i(U)\times g_j(V)\subset U\times V,$$
		$$\Psi_{ij}(U\times V)\cap\Psi_{i'j'}(U\times V)=(f_i(U)\times g_j(V))\cap (f_{i'}(U)\times g_{j'}(V))=\emptyset~~\text{for all } {(i,j)}\ne {(i',j')}.$$
		Therefore the IFS $\mathcal{F}\times \mathcal{G}$ satisfies OSC. The probability $p_iq_j$ correspond to similarity $\Psi_{ij}$. Then the quantization dimension $D_r$ of order $0<r<\infty$ of invariant measure $\mu_{\mathcal{F}\times \mathcal{G}}$ is given by
		$$\sum_{i=1,j=1}^{N}(p_iq_jc^r)^{\frac{D_r}{r+D_r}}=1.$$
		This implies that 
		\begin{align*}
			\sum_{i=1}^{N}(p_ic^r)^{\frac{D_r}{r+D_r}}\sum_{j=1}^{N}({q_j})^{\frac{D_r}{r+D_r}}=1.
		\end{align*}
		Since $\sum_{j=1}^{N}({q_j})^{\frac{D_r}{r+D_r}}>\sum_{i=1}^{N}q_j=1,$ we get
		$$\sum_{i=1}^{N}(p_ic^r)^{\frac{D_r}{r+D_r}}<1.$$
		
		Since $t\to  \sum_{i=1}^{N}(p_ic^r)^{t}$ is a strictly decreasing continuous function, by the above and Equation \eqref{e3.3}, we get
		\begin{align*}
			\frac{D_r}{r+D_r}&> \frac{{D_r}^*}{r+{D_r}^*}\\rD_r + D_r{D_r}^*&> r{D_r}^*+ D_r{D_r}^*\\
			D_r&>{D_r}^*.
		\end{align*}
		Similarly,  we can  also prove that $D_r>D_r'$. Therefore, 
		$$D_r > \max \{{D_r}^*,D_r'\}.$$ This completes the proof.
	\end{proof}
	\begin{remark}
		Let $\mathcal{F}=\{\mathbb{R}; f_1(x)=\frac{x}{2},f_2(x)=\frac{x}{2}+\frac{1}{2}\}$ and  $\mathcal{G}=\{\mathbb{R}; g_1(x)=\frac{x}{2},g_2(x)=\frac{x}{2}+\frac{1}{2} \}$ be two IFSs with the same probability vector $(\frac{1}{2},\frac{1}{2})$. Define the IFS $\mathcal{H}:=\{\mathbb{R}^2; h_1,h_2\}$, where $h_i(x,y)=(f_i(x),g_i(y)).$ In this case, we have
		\begin{itemize}
			\item $A_{\mathcal{F}}=A_{\mathcal{G}}=[0,1]$ and $A_{\mathcal{H}}= \{(x,x):x \in [0,1]\}$.
			\item $(\mu_{\mathcal{F}} \times \mu_{\mathcal{G}})( A_{\mathcal{H}})=0.$
		\end{itemize}
	\end{remark}
	
	By the above remark, it is clear that we may not obtain  the invariant measure of sub IFS by restricting the invariant measure of the super IFS.
	
	Next, we give an example of a measure supported on a countable set and also compute its quantization dimension.
	\begin{example}
		Let $E=\bigg\{x_m: x_m=\frac{1}{2}\bigg(\frac{1}{m}+\frac{1}{m+1}\bigg)~~~\forall~~m\in \mathbb{N}\bigg \}$ and let $\mu$ be the measure supported on $E$ such that $\mu(x_m)=\frac{\gamma}{m^2}$, where $\gamma=(\sum_{m=1}^{\infty}m^{-2})^{-1}.$ We shall show that $$D_r(\mu)=\frac{1}{2+ \frac{1}{r}}.$$
		\par 
		Let $A$ be a subset of $\mathbb{R}$ with Card$(A)=n$. We define a set
		$$\Delta= \bigg\{m\in \mathbb{N}: A\cap\bigg[\frac{1}{m+1},\frac{1}{m}\bigg]=\emptyset \bigg\}.$$
		Then, for each $m\in \Delta$, we have
		\begin{align*}
			\min_{a\in A}|x_m-a|^r\geq \bigg(\frac{\frac{1}{m}-\frac{1}{m+1}}{2}\bigg)^r=\frac{1}{2^r m^r (m+1)^r}. 
		\end{align*}
		So that,
		\begin{align*}
			\int\min_{a\in A}|x-a|^r d\mu(x) &\geq \sum_{m\in \Delta}\min_{a\in A}|x_m-a|^r\frac{\gamma}{m^2}\\
			&\geq \sum_{m\in \Delta}\frac{1}{2^r m^r (m+1)^r}\frac{\gamma}{m^2}\\
			&\geq \frac{\gamma}{2^r}\sum_{m=n+1}\frac{1}{(m+1)^{2r+2}}\\
			&\geq \frac{\gamma}{2^r} \int_{n+1}^{\infty}\frac{1}{(x+1)^{2r+2}} dx\\
			&=\frac{\gamma}{2^r}\frac{(n+1)^{-2r-1}}{(2r+1)}.
		\end{align*}
		From the above, we deduce that  $V_{n,r}(\mu)\geq \gamma_1 (n+1)^{-2r-1}$, where $\gamma_1= \frac{\gamma}{2^r (2r+1)}.$   Since $e_{n,r}(\mu)=V_{n,r}^{\frac{1}{r}}{\mu}$, we get 
		$$ e_{n,r}(\mu)\geq {\gamma_1}^\frac{1}{r} (n+1)^{-2-\frac{1}{r}}.$$
		From the above inequality, we conclude that 
		$$\underline{D}_r(\mu)=\liminf_{n\to \infty}\frac{\log n}{-\log e_{n,r}(\mu)}\geq \lim_{n\to \infty}\frac{\log n}{-\frac{1}{r}\log \gamma_1+(2+\frac{1}{r})\log (n+1)}=\frac{1}{2+\frac{1}{r}}.$$
		Using the same technique of \cite[Example 3.5]{Fal}, it is not difficult to compute that $\dim_B(E)=\frac{1}{2}.$ By using Lemma 1 in \cite{Zhu}, for any real number $t>\dim_B(E)$, we can choose a set $A_n\subset \mathbb{R}$ with Card$(A_n)=n$ such that
		$$\min_{a\in A_n}|x-a|^r\leq (\gamma_2)^r n^\frac{-r}{t}~~~\forall~~~x\in E,$$
		where $\gamma_2$ is some constant. We define a set $B_{2n}=\{x_m: 1\leq m\leq n\}\cup A_n.$ Thus, by using the above inequality and the definition of $V_{2n,r}(\mu)$, we have
		\begin{align*}
			V_{2n,r}(\mu)&\leq \int \min_{b\in B_{2n}}|x-b|^r d\mu(x) \\
			&=\sum_{m=1}^{\infty}\min_{b\in B_{2n}}|x_m-b|^r\frac{\gamma}{m^2}\\
			&\leq \gamma(\gamma_2)^r n^\frac{-r}{t} \sum_{m=n+1}^{\infty}\frac{1}{m^2} \\
			&\leq \gamma(\gamma_2)^r n^\frac{-r}{t} \int_{n}^{\infty} \frac{1}{x^2} dx\\
			&=\gamma_0~ n^{\frac{-r}{t}-1},
		\end{align*}
		where $\gamma_0=\gamma(\gamma_2)^r.$ By the above inequality and the definition of $e_{2n,r}(\mu),$ we deduce that
		$$e_{2n,r}{\mu}\leq {\gamma_0}^\frac{1}{r}  n^{\frac{-1}{t}-\frac{1}{r}}.$$
		Therefore, by the above inequality, we determine that
		$$\overline{D}_r(\mu)=\limsup_{n\to \infty} \frac{\log n}{-\log e_{n,r}(\mu)}\leq \lim_{n\to \infty} \frac{\log 2n}{-\frac{1}{r}\log \gamma_0+({\frac{1}{t}+\frac{1}{r}})\log n}=\frac{1}{{\frac{1}{t}+\frac{1}{r}}}.$$
		The above inequality holds for any $t>\frac{1}{2}.$ Therefore,  $\overline{D}_r(\mu)\leq \frac{1}{2+\frac{1}{r}}$. This proves our claim. 
	\end{example}
	\section{bounds on the quantization dimension of the invariant probability measure supported on the attractor of a bi-Lipschitz WIFS}\label{Se4}
	Firstly, we give some lemmas and propositions for determining an upper bound of the quantization dimension for the invariant Borel probability measure supported on the attractor of a bi-Lipschitz WIFS.
	\begin{lemma}\label{le3.13}
		Let $\mu$ be a Borel probability measure on $\mathbb{R}^{d}$ and $f:\mathbb{R}^{d}\to \mathbb{R}^{d}$ be a bi-Lipschitz map such that 
		$s\|x-y\|\leq\|f(x)-f(y)\|\leq c\|x-y\|,$ where $x,y\in \mathbb{R}^{d}$ and $s,c\in (0,1).$ Then 
		$$V_{n,r}(\mu\circ f^{-1} )\leq c^r V_{n,r}(\mu).$$
	\end{lemma}
	\begin{proof}
		Let $\nu=\mu\circ f^{-1}$. It can be easily observed that $\nu$ is a Borel probability measure on $\mathbb{R}^{d}$. We begin with the definition of $ V_{n,r}(\nu)$. We have
		\begin{align*}
			V_{n,r}(\nu)&=\inf \Big\{\int \min_{a\in A}\|x- a\|^r d\nu(x): A \subset \mathbb{R}^d, \, \text{Card}(A) \le n\Big\}\\
			&=\inf \Big\{\int  \min_{a\in A}\|x- a\|^rd(\mu\circ f^{-1})(x): A \subset \mathbb{R}^d, \, \text{Card}(A) \le n\Big\}\\
			&=\inf \Big\{\int  \min_{a\in A}\|f(y)- a\|^rd\mu(y): A \subset \mathbb{R}^d, \, \text{Card}(A) \le n\Big\}\\
			&\leq \inf \Big\{\int  \min_{f(b)\in f(A_1)}\|f(y)- f(b)\|^rd\mu(y): A_1 \subset \mathbb{R}^d, \, \text{Card}(A_1) \le n\Big\}\\
			&\leq c^r\inf \Big\{\int  \min_{b\in A_1}\|y- b\|^rd\mu(y): A_1 \subset \mathbb{R}^d, \, \text{Card}(A_1) \le n\Big\}\\
			&=c^rV_{n,r}(\mu).
		\end{align*}
		Thus, the proof of the lemma is established.
	\end{proof}
	\begin{lemma}\label{le3.14}
		Let $\mu_i$ be a Borel probability measure on $\mathbb{R}^d$  for $i=1,2,\cdots, K$. We define $\mu=\sum\limits_{i=1}^{K}s_i\mu_i$, where $s_i\geq0$ and $\sum\limits_{i=1}^{K}s_i=1.$ If $\sum\limits_{i=1}^{K}n_i\leq n$ where $n_i\in \mathbb{N}$ and $ \int \|x\|^r d\mu_i(x)<\infty$ for $i=1,2,\cdots, K$. Then
		$$V_{n,r}(\mu)\leq \sum\limits_{i=1}^{K}s_iV_{n_i,r}(\mu_i).$$
	\end{lemma}
	\begin{proof}
		Since $ \int \|x\|^r d\mu_i(x)<\infty$ for $i=1,2,\cdots, K$, we get $n_i$-optimal set $A_{n_i}\subset\mathbb{R}^d$ for measure $\mu_i$  with $\sum\limits_{i=1}^{K}n_i\leq n$. Let $A=\cup_{i=1}^{K}A_{n_i}$. Then Card$(A)\leq n$. By the definition of $V_{n,r}(\mu)$, we have 
		\begin{align*}
			V_{n,r}(\mu)&\leq \int \min_{a\in A}\|x- a\|^r d\mu(x) \\
			&= \sum_{i=1}^{K} s_i \int \min_{a\in A}\|x- a\|^r d\mu_i(x)\\
			&\leq  \sum_{i=1}^{K} s_i \int \min_{a\in A_{n_i}}\|x- a\|^r d\mu_i(x)\\
			&= \sum_{i=1}^{K} s_i V_{n_i,r}(\mu_i).
		\end{align*}
		This completes the proof.
	\end{proof}
	
	\begin{lemma} Let $\{\mathbb{R}^{d};f_1,f_2,\cdots,f_N\}$ be an IFS such that each $f_i$ satisfies
		$s_i\|x-y\|\leq\|f_i(x)-f_i(y)\|\leq c_i\|x-y\|$, where $x,y\in \mathbb{R}^d$ and $0<s_i\leq c_i<1.$ Also, assume that $(p_1,p_2,\cdots,p_N)$ be a probability vector corresponding to the IFS. 
		Let $r\in (0,\infty)$ be a fixed number. Then there exists a unique $l_r\in (0,\infty)$, satisfying 
		$$\sum_{i=1}^{N}(p_i{c_i}^r)^{\frac{l_r}{r+l_r}}=1.$$
	\end{lemma}
	\begin{proof}
		We define a function $F:[0,\infty)\to \mathbb{R}$ by
		$$F(t)=\sum_{i=1}^{N}(p_i{c_i}^r)^t.$$ It can be easily observe that $F$ is a strictly decreasing continuous function as $0<p_i,c_i<1$ and $r\in(0,\infty)$. Furthermore, $F(0)=N\geq 2$ and $F(1)=\sum_{i=1}^{N}p_ic_i^{r}<\sum_{i=1}^{N}p_i=1$. Therefore, by the intermediate value theorem and using strictly decreasing property of $F$, there exists a unique number $t_0\in (0,1)$ such that $F(t_0)=\sum_{i=1}^{N}(p_i{c_i}^r)^{t_0}=1.$  Since $t_0\in (0,1)$, there is a unique $l_r\in (0,\infty )$ such that $\frac{l_r}{r+l_r}=t_0.$ Hence, there is a unique number $l_r\in (0,\infty)$ such that   $\sum_{i=1}^{N}(p_i{c_i}^r)^{\frac{l_r}{r+l_r}}=1.$ This completes the proof.
	\end{proof}
	\begin{proposition}\label{le3.18}
		Let $\mathcal{W}=\{\mathbb{R}^{d};~f_1,f_2,\cdots,f_N;~~p_1,p_2,\cdots,p_N\}$ be a weighted IFS (WIFS) such that each $f_i$ satisfies
		$s_i\|x-y\|\leq\|f_i(x)-f_i(y)\|\leq c_i\|x-y\|$, where $x,y\in \mathbb{R}^d$ and $0<s_i\leq c_i<1$. Let $\mu$ be the invariant Borel probability measure corresponding to the WIFS $\mathcal{W}$. Then for every $n\in \mathbb{N}$ and $ r\in (0,\infty)$
		$$V_{n,r}(\mu)\leq \min \bigg\{\sum_{i=1}^{N}p_i{c_i}^rV_{n_i,r}(\mu):~~n_i\in \mathbb{N}, \sum\limits_{i=1}^{N}n_i\leq n\bigg\}.$$
	\end{proposition}
	\begin{proof}
		Since $\mu$ is the invariant measure corresponding to the WIFS $\mathcal{W}$,  $\mu=\sum\limits_{i=1}^{N}p_i\mu\circ {f_i}^{-1}.$ Using Lemma \ref{le3.13}, for each $n\in\mathbb{N}$ we get 
		$$V_{n,r}(\mu\circ {f_i}^{-1} )\leq c_i^r V_{n,r}(\mu)~~~\forall~~i=1,2,\cdots,N.$$
		Since $ \int \|x\|^r d\mu(x)<\infty$,   $ \int \|x\|^r d(\mu\circ {f_i}^{-1})(x)<\infty$. Thus, in the light of Lemma \ref{le3.14} and by the above inequality, we get 
		$$V_{n,r}(\mu)\leq \min \bigg\{\sum_{i=1}^{N}p_i{c_i}^rV_{n_i,r}(\mu):~~n_i\in \mathbb{N}, \sum\limits_{i=1}^{N}n_i\leq n\bigg\}.$$
		Thus, the proof is complete.
	\end{proof}
	\begin{corollary}\label{co3.19}
		Let $\mathcal{W}=\{\mathbb{R}^{d};~f_1,f_2,\cdots,f_N;~~p_1,p_2,\cdots,p_N\}$ be a WIFS such that each $f_i$ satisfies
		$s_i\|x-y\|\leq\|f_i(x)-f_i(y)\|\leq c_i\|x-y\|$, where $x,y\in \mathbb{R}^d$ and $0<s_i\leq c_i<1$ and each $p_i>0$. Let $\mu$ be the invariant Borel probability measure corresponding to the WIFS $\mathcal{W}$.
		Let $\Gamma\subset \{1,2,\cdots,N\}^*$ be a  finite maximal antichain. Then for any $n\in \mathbb{N}$ with $n\geq |\Gamma|$, $\sigma\in \Gamma$ and $r\in(0,\infty)$ 
		$$V_{n,r}(\mu)\leq \min \bigg\{\sum_{\sigma\in \Gamma}p_\sigma c_{\sigma}^rV_{n_\sigma,r}(\mu):~~n_\sigma\in \mathbb{N}, \sum\limits_{\sigma\in \Gamma}n_\sigma\leq n\bigg\}.$$
	\end{corollary}
	\begin{proof}
		Since $\Gamma$ is a finite maximal antichain, by Lemma \ref{le3.15} measure $\mu$  satisfies 
		$$\mu=\sum\limits_{\sigma\in \Gamma}p_\sigma\mu\circ {f_\sigma}^{-1}.$$
		The rest of the proof of this corollary follows from the proof of Proposition \ref{le3.18}. 
	\end{proof}
	In the upcoming theorem, we give an upper bound of the quantization dimension for the invariant Borel probability measure supported on the attractor of a bi-Lipschitz WIFS without any separation condition on the IFS.
	\begin{theorem}\label{themlower}
		Let  $\mathcal{W}=\{\mathbb{R}^{d};~f_1,f_2,\cdots,f_N;~~p_1,p_2,\cdots,p_N\}$ be a WIFS such that each $f_i$ satisfies
		$s_i\|x-y\|\leq\|f_i(x)-f_i(y)\|\leq c_i\|x-y\|$, where $x,y\in \mathbb{R}^d$ and $0<s_i\leq c_i<1$ and each $p_i>0$. Let $\mu$ be the invariant Borel probability measure corresponding to the WIFS $\mathcal{W}$.
		Let $r\in (0,\infty)$ and $l_r\in (0,\infty)$ be the unique number such that
		$\sum\limits_{i=1}^{N}(p_i{c_i}^r)^{\frac{l_r}{r+l_r}}=1.$ Then
		$$\limsup_{n\to \infty}n\cdot e_{n,r}^{l_r}(\mu)<\infty,$$
		where $e_{n,r}(\mu)={V_{n,r}(\mu)}^{\frac{1}{r}}.$ Moreover $\overline {D_r}(\mu)\leq l_r,$ where $\overline {D_r}(\mu)$ denotes the upper quantization dimension of order $r$ of $\mu.$
	\end{theorem}
	\begin{proof}
		Let $\xi_i=(p_i{c_i}^r)^{\frac{l_r}{r+l_r}}$. Then each $\xi_i>0$ and $\sum\limits_{i=1}^{N}\xi_i=1$. Therefore $(\xi_1,\xi_2,\cdots,\xi_N)$ is a probability vector. Let $\xi_{min}=\min\{\xi_1,\xi_2,\cdots,\xi_N\}$. Then we have $\xi_{min}>0$. Let $m_0 \in \mathbb{N}$ be fixed. Choose $n\in \mathbb{N}$ with the property $\frac{m_0}{n}<\xi_{min}^2$, this property holds for for all but finitely many values of $n\in \mathbb{N}.$ Set $\epsilon=\xi_{min}^{-1}\frac{m_0}{n}.$
		We define a set 
		$$\Gamma_\epsilon=\big\{\sigma\in \{1,2,\cdots,N\}^*;~~ \xi_{\sigma^-}\geq \epsilon >\xi_{\sigma}\big\}.$$
		Lemma \ref{le3.16} yields that $\Gamma_\epsilon$ is a finite maximal antichain. Therefore, in the light of Lemma \ref{le3.15}, we obtain 
		$$ 1=\sum_{\sigma \in \Gamma_{\epsilon}}\xi_{\sigma}
		=\sum_{\sigma \in \Gamma_{\epsilon}}\xi_{\sigma^{-}}\cdot \xi_{\sigma_{|\sigma|}}
		\geq \sum_{\sigma \in \Gamma_{\epsilon}}\epsilon\cdot \xi_{\sigma_{|\sigma|}}
		\geq \sum_{\sigma \in \Gamma_{\epsilon}}\epsilon\cdot \xi_{min}
		=\epsilon\cdot \xi_{min}\cdot |\Gamma_\epsilon|.$$
		
		Therefore, by the above, we get
		$ |\Gamma_{\epsilon}|\leq (\epsilon~\cdot \xi_{min})^{-1}=\frac{n}{m_0}.$ This can be also written as $\sum\limits_{\sigma\in \Gamma_{\epsilon}}m_0\leq n.$ So, by using Corollary \ref{co3.19}, we get the following inequality
		\begin{align*}
			V_{n,r}(\mu)&\leq \sum\limits_{\sigma\in \Gamma_{\epsilon}}p_\sigma c_{\sigma}^{r} V_{m_0,r}(\mu)\\
			&=\sum\limits_{\sigma\in \Gamma_{\epsilon}}(p_\sigma c_{\sigma}^{r})^{\frac{l_r}{r+l_r}}(p_\sigma c_{\sigma}^{r})^{\frac{r}{r+l_r}} V_{m_0,r}(\mu)\\
			&=\sum\limits_{\sigma\in \Gamma_{\epsilon}}\xi_\sigma (p_\sigma c_{\sigma}^{r})^{\frac{r}{r+l_r}} V_{m_0,r}(\mu).
		\end{align*}
		Since $\epsilon>\xi_\sigma=(p_\sigma{c_\sigma}^r)^{\frac{l_r}{r+l_r}}$ for $\sigma \in \Gamma_\epsilon,$ Therefore, by the above inequality, we conclude that
		\begin{align*}
			V_{n,r}(\mu)&\leq \sum_{\sigma\in \Gamma_\epsilon}\xi_{\sigma}\bigg({\epsilon^{\frac{r+l_r}{l_r}}}\bigg)^{\frac{r}{r+l_r}} V_{m_0,r}(\mu)\\
			&=\sum_{\sigma\in \Gamma_\epsilon}\xi_{\sigma}\cdot{\epsilon}^{\frac{r}{l_r}} \cdot V_{m_0,r}(\mu)\\
			&={\epsilon}^{\frac{r}{l_r}} \cdot V_{m_0,r}(\mu)\sum_{\sigma\in \Gamma_\epsilon}\xi_{\sigma}.
		\end{align*}
		Since $\epsilon={\xi_{\min}}^{-1}\frac{m_0}{n}$ and by Lemma \ref{le3.15}, we get
		\begin{align*}
			V_{n,r}(\mu)&\leq \xi_{min}^{\frac{-r}{l_r}}\bigg({\frac{m_0}{n}}\bigg)^{\frac{r}{l_r}}V_{m_0,r}(\mu)\\
			n^{\frac{r}{l_r}}\cdot V_{n,r}(\mu)&\leq \xi_{min}^{\frac{-r}{l_r}}\cdot{m_0}^{\frac{r}{l_r}}\cdot V_{m_0,r}(\mu)\\
			n \cdot e_{n,r}^{l_r}(\mu)&\leq \xi_{min}^{-1}\cdot m_0 \cdot e_{m_0,r}^{l_r}(\mu).
		\end{align*}
		The above inequality holds for all but finitely many values of n. Therefore, we get 
		$$\limsup_{n\to\infty} n\cdot e_{n,r}^{l_r}(\mu)\leq \xi_{min}^{-1}\cdot m_0 \cdot e_{m_0,r}^{l_r}(\mu)<\infty.$$
		By using Proposition \ref{prop2.12}, we get $\overline{D}_r(\mu)\leq l_r.$ 
		Thus, the proof of the theorem is done.
	\end{proof}
	Next, for obtaining a lower bound of the quantization  dimension of invariant Borel probability measure corresponding to a WIFS, we give some lemmas and propositions. 
	\begin{lemma}\label{le3.21}
		Let $\mathcal{W}=\{\mathbb{R}^{d};~f_1,f_2,\cdots,f_N;~~p_1,p_2,\cdots,p_N\}$ be a WIFS such that each $f_i$ satisfies
		$s_i\|x-y\|\leq\|f_i(x)-f_i(y)\|\leq c_i\|x-y\|$, where $x,y\in \mathbb{R}^d$ and $0<s_i\leq c_i<1$.  Let $\mu$ be the invariant Borel probability measure and $A_{\mathcal{W}}$ be the invariant attractor corresponding to the WIFS $\mathcal{W}$. Then for every $\epsilon>0$
		$$\inf\{\mu(B(x,\epsilon)): x\in A_{\mathcal{W}}\}>0,$$
		where $B(x,\epsilon)$ is the open ball of radius $\epsilon $ and center $x$.
	\end{lemma}
	\begin{proof}
		It is  well-known that there is a unique $m\in \mathbb{R}$ such that $\sum\limits_{i=1}^{N}{c_i}^{m}=1$. Set $q_i={c_i}^m$. Then $(q_1,q_2,\cdots,q_N)$ is a probability vector. Let $\epsilon_0=\big(\frac{\epsilon}{\text{diam}A_\mathcal{W}}\big)^{m}.$ Then $\Gamma_{\epsilon_0}=\{\sigma\in\{1,2,\cdots,N\}^*:~  q_{{\sigma}^-}\geq \epsilon_0 > q_\sigma\}$ is a finite maximal antichain. Let $x\in A_\mathcal{W}$. Since $A_{\mathcal{W}}=\cup_{\sigma\in \Gamma_{\epsilon_0}}f_\sigma(A_{\mathcal{W}})$, there is $\sigma_0\in \Gamma_{\epsilon_0}$ such that $x\in f_{\sigma_0}(A_\mathcal{W}).$ Let $x_1,x_2\in A_\mathcal{W}.$ Then
		$$\|f_{\sigma_0}(x_1)-f_{\sigma_0}(x_2)\| \leq c_{\sigma_0}\|x_1-x_2\|.$$
		From the above inequality, we can conclude that  $\text{diam}(f_{\sigma_0}(A_{\mathcal{W}}))\leq c_{\sigma_0}~\text{diam}(A_{\mathcal{W}})$. Thus \\ $ \text{diam}(f_{\sigma_0}(A_{\mathcal{W}}))\leq \epsilon $, which further yields that $f_{\sigma_0}(A_{\mathcal{W}})\subseteq B(x,\epsilon)$. Since $\mu$ satisfies $\mu=\sum\limits_{\sigma\in \Gamma_{\epsilon_0}}p_\sigma\mu\circ {f_\sigma}^{-1}$,  we have
		\begin{align*}
			\mu(f_{\sigma_0}(A_{\mathcal{W}}))&=\sum\limits_{\sigma\in \Gamma_{\epsilon_0}}p_\sigma\mu\circ {f_\sigma}^{-1}(f_{\sigma_0}(A_{\mathcal{W}}))\\
			&\geq p_{\sigma_0}\mu\circ {f_{\sigma_0}}^{-1}(f_{\sigma_0}(A_{\mathcal{W}}))\\
			&=p_{\sigma_0}.
		\end{align*}
		Let $p=\min\{p_\sigma;~~\sigma\in \Gamma_{\epsilon_0}\}>0$. Then by the above inequality $$\mu(B(x,\epsilon))\geq \mu(f_{\sigma_0}(A_{\mathcal{W}}))\geq p>0.$$
		This holds for arbitrary $x\in A_\mathcal{W}.$ Therefore, we get our assertion. 
	\end{proof}
	\begin{lemma}\label{lem3.22}
		Let $\mathcal{W}$ as in the above lemma satisfy strong open set condition and  $\mu $ be the invariant measure corresponding to $\mathcal{W}.$ Then for each $m\in \mathbb{N}$, there is a WIFS $\mathcal{L}_m=\{\mathbb{R}^d;~~g_i~; ~p_i~~\forall~~i\in \{1,2,\cdots,N\}^m\}$  such that $S_i\|x-y\|\leq\|g_i(x)-g_i(y)\|\leq C_i\|x-y\|$, where $x,y\in \mathbb{R}^d$ and $0<S_i\leq C_i<1$ and the  WIFS $\mathcal{L}_m$ satisfies SSC. Furthermore if $\mu_m^*$ is the invariant Borel probability measure corresponding to the IFS $\mathcal{L}_m$, then there exists $n_0\in \mathbb{N}$ such that for each $n\geq n_0$ there exists $n_i\in \mathbb{N}$ such that $\sum\limits_{i\in \{1,2,\cdots,N\}^m}n_i\leq n$ and
		$$V_{n,r}(\mu_m^*)\geq  \sum_{i\in \{1,2,\cdots,N\}^m}p_i{S_i}^rV_{n_i,r}(\mu_m^*),$$
		for any  $r\in (0,\infty)$.
	\end{lemma}
	\begin{proof}
		Since $\mathcal{W}$ satisfies the strong open set condition, there exists an open set $U$ such that 
		$$\cup_{i=1}^Nf_i(U) \subset U,~~U\cap A_\mathcal{W}\ne\emptyset,
		f_i(U) \cap f_j(U)= \emptyset ~~ \text{for all}~ i \ne j.$$
		Since $U\cap A_\mathcal{W}\ne\emptyset,$ there exists $\sigma\in \{1,2,\cdots,N\}^*$ such that $f_\sigma(A_\mathcal{W})\subseteq U$. Set $g_i=f_i\circ f_\sigma$ for $ i\in \{1,2,\cdots,N\}^m$. For $x,y\in \mathbb{R}^d$, we have 
		$$ s_is_\sigma\|x-y\|\leq\|g_i(x)-g_i(y)\|\leq c_ic_\sigma\|x-y\|.$$ Hence 
		$$ S_i\|x-y\|\leq\|g_i(x)-g_i(y)\|\leq C_i\|x-y\|,$$
		where $S_i= s_is_\sigma\in (0,1)$ and $C_i=c_ic_\sigma\in (0,1).$ Therefore $\mathcal{L}_m=\{\mathbb{R}^d;~~g_i~~~ \forall~~i\in \{1,2,\cdots,N\}^m\}$ is an IFS.  Since  $\sum\limits_{i\in \{1,2,\cdots,N\}^m}p_{i\sigma}=p_\sigma\ne 1$, we have  $\sum\limits_{i\in \{1,2,\cdots,N\}^m}\frac{p_{i\sigma}}{p_\sigma}=\sum\limits_{i\in \{1,2,\cdots,N\}^m}p_i= 1$. 
		Therefore $\mathcal{L}_m=\{\mathbb{R}^d;~~g_i~; ~p_i~~\forall~~i\in \{1,2,\cdots,N\}^m\}$ is a WIFS.
		Let $A_{{m}}^*$ be the invariant attractor and let $\mu_m^*$ be the invariant Borel probability measure corresponding to WIFS $\mathcal{L}_m$.  Analysing the code space of both IFS $\mathcal{W}$ and $\mathcal{L}_m,$ one can  easily show that $A_{m}^*\subseteq A_\mathcal{W}$. Since $f_i(U) \cap f_j(U)= \emptyset ~~ \text{for all}~ i \ne j\in \{1,2,\cdots,N\}^m$ and $f_\sigma(A_\mathcal{W})\subseteq U$, we get $g_{i}(A_{m}^*) \cap g_{j}(A_{m}^*)= \emptyset ~~ \text{for all}~ i \ne j\in \{1,2,\cdots,N\}^m$. Thus the WIFS  $\mathcal{L}_m=\{\mathbb{R}^d;~~g_i~; ~p_i~~\forall~~i\in \{1,2,\cdots,N\}^m\}$ satisfies strong separation condition. 
		Set $\delta_0=\min\limits_{i\ne j\in \{1,2,\cdots,N\}^m}\bigg\{\min\limits_{a\in g_i(A_{m}^*)}\min\limits_{b\in g_j(A_{m}^*)}\|a-b\|\bigg\}$. Since the WIFS $\mathcal{L}_m$ satisfies strong separation condition, we have $\delta_0>0$. Lemma \ref{le3.21} yields that
		$$\delta= \inf\bigg\{\bigg(\frac{\delta_0}{8}\bigg)^r\mu_m^*\bigg(B\bigg(x,\frac{\delta_0}{8}\bigg)\bigg):~~x\in A_{m}^*\bigg\}>0.$$
		By using Lemma \ref{le2.9}, we have $V_{n,r}(\mu_m^*)\to 0$ as $n\to \infty$. Then there is an $n_0\in \mathbb{N}$ such that $V_{n,r}(\mu_m^*)<\delta $ for all $n\geq n_0.$
		Let $A_n$ be an $n$-optimal set for measure $\mu_m^*$ of order $r$. In the light of Lemma \ref{le2.11}, for all $n\geq n_0$, we have 
		$$\bigg(\frac{\|A_n\|_{\infty}}{2}\bigg)^r\min\limits_{x\in A_{m}^*}\mu_m^*\bigg(B\bigg(x,\frac{\|A_n\|_{\infty}}{2}\bigg)\bigg)<\delta.$$
		It can be easily shown that for $r\in (0,\infty)$, function $\phi: (0,\infty)\to \mathbb{R}$ defined by $\phi(t)=t^r\min\limits_{x\in A_{m}^*}\mu_m^*(B(x,t))$ is an increasing function. This further yields that
		$$\frac{\|A_n\|_{\infty}}{2}<\frac{\delta_0}{8}$$ and hence 
		$$\|A_n\|_{\infty}<\frac{\delta_0}{4}$$
		for all $n\geq n_0$. We define $ A_{n_i}=\{a\in A_n; ~~ W(a|A_n)\cap g_i(A_{m}^*)\ne\emptyset\}$ for $i\in \{1,2,\cdots,N\}^m$ and let $n_i$ be the cardinality of $A_{n_i}$. By the definition of $\|A_n\|_\infty$ and $\delta_0$, we conclude that for all $n\geq n_0$, $A_{n_i}\cap A_{n_j}=\emptyset$ for $i\ne j\in\{1,2,\cdots,N\}^m $ and $\sum\limits_{i\in \{1,2,\cdots,N\}^m }n_i\leq n.$ Now for $n\geq n_0$, we have
		\begin{align*}
			V_{n,r}(\mu_m^*)&=\int\min_{a\in A_n}\|x-a\|^r d\mu_m^*(x)\\
			&=\sum_{i\in \{1,2,\cdots,N\}^m}p_i\int\min_{a\in A_n}\|g_i(x)-a\|^r d\mu_m^*(x)\\
			&=\sum_{i\in \{1,2,\cdots,N\}^m}p_i\int\min_{a\in A_{n_i}}\|g_i(x)-a\|^r d\mu_m^*(x)\\
			&\geq \sum_{i\in \{1,2,\cdots,N\}^m}p_i{S_i}^r\int\min_{b\in {g_i}^{-1}(A_{n_i})}\|x-b\|^r d\mu_m^*(x)\\
			&\geq \sum_{i\in \{1,2,\cdots,N\}^m}p_i{S_i}^r V_{n_i,r}(\mu_m^*).
		\end{align*}
		This completes the result.
	\end{proof} 
	\begin{proposition}\label{new917}
		Let $\mathcal{W}$ and $\mathcal{L}_m$ be two WIFSs as defined in the above lemma. For $r \in (0,\infty)$, let $k_{m,r}\in (0,\infty)$ be the unique number such that $\sum\limits_{i\in \{1,2,\cdots,N\}^m}({p_i{S_i}^r})^\frac{k_{m,r}}{r+k_{m,r}}=1.$ Then $$\liminf_{n\to \infty}n\cdot e_{n,r}^{l_0}(\mu_m^*)>0$$ for any $l_0\in (0,k_{m,r})$. Moreover $k_{m,r}\leq \underline{D}_r(\mu_m^*)$, where $\underline{D}_r(\mu_m^*)$ denotes the lower quantization dimension of measure $\mu_m^*$ of order $r$.
	\end{proposition}
	\begin{proof}
		Using the fact that $ t\longmapsto \sum\limits_{i\in\{1,2,\cdots,N\}^m}(p_i{S_i}^r)^t$ is a strictly decreasing function  and \\ $\sum\limits_{i\in \{1,2,\cdots,N\}^m}(p_i{S_i}^r)^{\frac{k_{m,r}}{r+k_{m,r}}}=1,$ then for $0<l_0<k_{m,r}$, we have 
		$$\sum\limits_{i\in\{1,2,\cdots,N\}^m }(p_i{S_i}^r)^{\frac{l_0}{r+l_0}}>\sum\limits_{i\in \{1,2,\cdots,N\}^m}(p_i{S_i}^r)^{\frac{k_{m,r}}{r+k_{m,r}}}=1.$$
		Let $n_0$ be as in the above lemma. Let $C_*=\min\{n^{\frac{r}{l_0}}\cdot V_{n,r}(\mu_m^*); ~~n<n_0\}.$ It is easy to show that $V_{n,r}(\mu_m^*)>0$, from which we deduce that $C_*>0.$ Our aim is to show that $n^{\frac{r}{l_0}}\cdot V_{n,r}(\mu_m^*)\geq C_*$ for all $n\in\mathbb{N}$. Clearly the inequality hold for $n<n_0$ by the definition of $C_*$. For $n\geq n_0$, we prove it by induction on $n\in \mathbb{N}.$ Let $n\geq n_0$ and $\eta^{\frac{r}{l_0}}\cdot V_{\eta,r}(\mu_m^*)\geq C_*$ hold for all $\eta<n.$ Lemma \ref{lem3.22} yields that there are numbers $n_i\in \mathbb{N}$ such that $\sum\limits_{i\in \{1,2,\cdots,N\}^m}n_i\leq n$ and
		\begin{align*}
			n^{\frac{r}{l_0}} V_{n,r}(\mu_m^*)&\geq  n^{\frac{r}{l_0}}  \sum_{i\in \{1,2,\cdots,N\}^m}p_i{S_i}^rV_{n_i,r}(\mu_m^*)\\
			&= n^{\frac{r}{l_0}} \sum_{i\in \{1,2,\cdots,N\}^m}p_i{S_i}^r {n_i}^{\frac{-r}{l_0}} {n_i}^{\frac{r}{l_0}} V_{n_i,r}(\mu_m^*)\\
			&\geq C_* n^{\frac{r}{l_0}} \sum_{i\in \{1,2,\cdots,N\}^m}p_i{S_i}^r {n_i}^{\frac{-r}{l_0}}\\
			&= C_* \sum_{i\in \{1,2,\cdots,N\}^m}p_i{S_i}^r {\bigg(\frac{n_i}{n}}\bigg)^{\frac{-r}{l_0}}.
		\end{align*}
		Using H\"older's inequality for negative exponents, we get
		$$n^{\frac{r}{l_0}} V_{n,r}(\mu_m^*)\geq C_*\bigg(\sum\limits_{i\in \{1,2,\cdots,N\}^m}(p_i{S_i}^r)^{\frac{l_0}{r+l_0}}\bigg)^{1+\frac{r}{l_0}}\bigg(\sum\limits_{i\in \{1,2,\cdots,N\}^m}{\bigg(\frac{n_i}{n}}\bigg)^{\frac{-r}{l_0}\cdot \frac{-l_0}{r}}\bigg)^{\frac{-r}{l_0}}.$$
		Since $\sum\limits_{i\in\{1,2,\cdots,N\}^m}(p_i{S_i}^r)^{\frac{l_0}{r+l_0}}>1 $ and $\sum\limits_{i\in \{1,2,\cdots,N\}^m}\frac{n_i}{n}\leq 1$, we have 
		$$n^{\frac{r}{l_0}} V_{n,r}(\mu_m^*)\geq C_*.$$
		Therefore, by induction on $n\in \mathbb{N}$, $n^{\frac{r}{l_0}} V_{n,r}(\mu_m^*)\geq C_*$ hold for all $n\in \mathbb{N}.$ Thus, we deduce that 
		$$\liminf_{n\to \infty}n\cdot e_{n,r}^{l_0}(\mu_m^*)\geq {C_*}^{\frac{l_0}{r}}>0.$$
		Proposition \ref{prop2.12} implies that $k_{m,r}\leq \underline{D}_r(\mu_m^*).$
		The proof is complete.
	\end{proof}
	In the following theorem, we determine a lower bound of the quantization of the invariant Borel probability measure supported on the attractor of a bi-Lipschitz WIFS
	under the SOSC.
	\begin{theorem}\label{themupper}
		Let $\mathcal{W}$ be the WIFS  defined as in Lemma \ref{le3.21} and assume that it satisfies the SOSC. For $r \in (0,\infty)$, we have $\underline{D}_r(\mu) \ge k_r,$ where $\sum\limits_{i=1}^{N}({p_i{s_i}^r})^\frac{k_r}{r+k_r}=1.$
	\end{theorem}
	\begin{proof}
		Since the IFS $\mathcal{W}$ satisfies the SOSC, there exists an open set $U$  of $\mathbb{R}^d$ such that    \[\cup_{i=1}^N f_i(U) \subset U,~~U\cap A_\mathcal{W}\ne\emptyset,
		f_i(U) \cap f_j(U)= \emptyset ~~ \forall~ i \ne j,~~1\leq i,j\leq N.\]
		Since $U\cap A_{\mathcal{W}}\ne \emptyset$, there exists a $\sigma \in \{1,2,\cdots,N\}^*$ such that $f_\sigma(A_\mathcal{W})\subset U$. We denote  $ f_\sigma(A_\mathcal{W})$ by $ (A_\mathcal{W})_\sigma$ for any $\sigma\in \{1,2,\cdots,N\}^*$. Now, by using the condition $f_i(U) \cap f_j(U)= \emptyset ~~ \forall~ i \ne j,~~1\leq i,j\leq N$, it is clear that for each $m\in \mathbb{N}$, the sets $\{(A_\mathcal{W})_{i\sigma}: i \in \{1,2,\cdots,N\}^{m} \}$ are pairwise disjoint. We define a WIFS $\mathcal{L}_m=\{\mathbb{R}^d;f_{i\sigma};~~p_i~~\forall~~i \in \{1,2,\cdots,N\}^m\}$ as in Lemma \ref{lem3.22}.  Therefore, by the Proposition \ref{new917}, we obtain that $ k_{m,r} \le \underline{D}_r(\mu_m^*)$, where $k_{m,r}$ is given by  $ \sum_{ i \in \{1,2,\cdots,N\}^m}(p_i s_{i\tau}^r)^{\frac{ k_{m,r}}{r+k_{m,r}}} =1.$ Let $A_m^*$ be the attractor of the IFS $\mathcal{L}_m$. Note that (Cf. \cite[Theorem $11.6$ ]{GL1}) $0< \dim_H(\mu) \le \underline{D}_r(\mu).$ Hence, $ \underline{D}_r(\mu)> 0.$ Now, we claim that $\underline{D}_r(\mu_m^*) \le \underline{D}_r(\mu) $.
		If  $\underline{\dim}_B(A_m^*)\le \underline{D}_r(\mu)$ is not true, then we may choose another word $\tau =\sigma \omega\in \{1,2,\cdots,N\}^*$ for some $\omega \in \{1,2,\cdots,N\}^*$ such that the attractor $A_m^*$ of the  WIFS $\mathcal{L}_m=\{\mathbb{R}^d:f_{i\tau};~~p_i~~\forall~~i \in \{1,2,\cdots,N\}^m\}$ satisfies $\underline{\dim}_B(A_m^*)\le \underline{D}_r(\mu)$. 
		Now, using \cite[Proposition $11.9$ ]{GL1}, that is, $\underline{D}_r(\mu_m^*) \le \underline{\dim}_B(A_m^*)$ and the condition $\underline{\dim}_B(A_m^*)\le \underline{D}_r(\mu)$, we obtain $\underline{D}_r(\mu_m^*) \le \underline{D}_r(\mu) $. Since  $ k_{m,r} \le \underline{D}_r(\mu_m^*)$, we have $k_{m,r}\leq \underline{D}_r(\mu).$ Now, we will show that $k_r\leq \underline{D}_r(\mu).$
		Suppose for contradiction that $ \underline{D}_r(\mu)< k_r.$ Let $ t_{max}=\max\{p_1s_1^r, p_2s_2^r, \dots,p_{N}s_N^r\}.$ Using $\sum_{ i \in \{1,2,\cdots,N\}^m} (p_is_i^r)^{\frac{k_r}{r+k_r}}=1$ and $\frac{\underline{D}_r(\mu)}{r+\underline{D}_r(\mu)} -\frac{k_r}{r+k_r} < 0$, we have
		\begin{equation*}
			\begin{aligned}
				s_{\tau}^{\frac{- r\cdot k_{m,r}}{r+k_{m,r}}}  &= \sum_{ i \in \{1,2,\cdots,N\}^m} (p_is_i^r)^{\frac{ k_{m,r}}{r+k_{m,r}}}\\ & \ge \sum_{ i \in \{1,2,\cdots,N\}^m} (p_is_i^r)^{\frac{ \underline{D}_r(\mu)}{r+\underline{D}_r(\mu)}} \\&  = \sum_{ i \in \{1,2,\cdots,N\}^m} (p_is_i^r)^{\frac{\underline{D}_r(\mu)}{r+\underline{D}_r(\mu)}}(p_is_i^r)^{\frac{- k_r}{r+k_r}}(p_is_i^r)^{\frac{k_r}{r+k_r}}\\&\ge  \sum_{i \in \{1,2,\cdots,N\}^m} (p_is_i)^{\frac{k_r}{r+k_r}} t_{max}^{m\Big(\frac{\underline{D}_r(\mu)}{r+\underline{D}_r(\mu)} -\frac{k_r}{r+k_r}\Big)}\\& \geq t_{max}^{m\Big(\frac{\underline{D}_r(\mu)}{r+\underline{D}_r(\mu)} -\frac{k_r}{r+k_r}\Big)} 
			\end{aligned}
		\end{equation*}
		This implies that $$s_{\tau}^{\frac{-r k_{r}}{r+k_{r}}} \geq t_{max}^{m\Big(\frac{\underline{D}_r(\mu)}{r+\underline{D}_r(\mu)} -\frac{k_r}{r+k_r}\Big)}. $$ 
		We have a contradiction for large values of $m\in \mathbb{N} $. Therefore, we get $ \underline{D}_r(\mu) \ge  k_r,$ proving the assertion.
	\end{proof}
	
	In the next theorem, we combine above results of the quantization dimension.
	\begin{theorem}\label{thmbilip}
		Let $\mathcal{W}=\{\mathbb{R}^{d};~f_1,f_2,\cdots,f_N;~~p_1,p_2,\cdots,p_N\}$ be a WIFS such that each $f_i$ satisfies
		\begin{align}\label{eqbi}
			s_i\|x-y\|\leq\|f_i(x)-f_i(y)\|\leq c_i\|x-y\|,
		\end{align}where $x,y\in \mathbb{R}^d$ and $0<s_i\leq c_i<1$. Let $\mu$ be the invariant Borel probability measure corresponding to the WIFS $\mathcal{W}$  and assume that $\mathcal{W}$ satisfies SOSC. Then 
		$$k_r\leq \underline{D}_r(\mu)\leq \overline{D}_r(\mu)\leq l_r,$$ where $k_r$ and $l_r$ are given by $\sum\limits_{i=1}^{N}({p_i{s_i}^r})^\frac{k_r}{r+k_r}=1$ and $\sum\limits_{i=1}^{N}({p_i{c_i}^r})^\frac{l_r}{r+l_r}=1$, respectively.
	\end{theorem}
	\begin{proof}
		By combining Theorems \ref{themlower} and \ref{themupper}, we get our required result.
	\end{proof}
	\begin{remark}
		In \cite{GL2}, Graf and Luschgy gave a formula of the quantization dimension of the invariant self-similar probability measure generated by a self-similar IFS under the open set condition. Here, we give bounds for the quantization dimension of invariant probability measure generated by bi-Lipschitz IFS under the strong open set condition. If we choose $s_i=c_i$ in \eqref{eqbi}, then our result will give the formula of Graf and Luschgy  \cite{GL2}. Thus, our result generalizes the result of  Graf and Luschgy  \cite{GL2} in more general setting.
	\end{remark}
	\begin{remark}
		In \cite{R2}, Roychowdhury determined the quantization dimension of the invariant probability measure supported on the limit set generated by a bi-Lipschitz  IFS with the  strong open set condition and the bi-Lipschitz constants satisfying $\overline{s}_\sigma\leq K \underline{s}_\sigma~~\forall~~ \sigma\in I^*.$ The condition taken by the author is true for similarity mapping but not for the general class of bi-Lipschitz mappings. For example, if we take $\overline{s}_i=\frac{1}{2}, \underline{s}_i=\frac{1}{3}$ for all $1\leq i\leq N,$ then we  cannot find any such $K\in \mathbb{R}$. However, our result Theorem \ref{thmbilip} gives the quantization dimensions of invariant probability measures corresponding to a general class of bi-Lipschitz IFSs.
	\end{remark}
	In the next proposition, we estimate the quantization dimension of the invariant Borel probability measures supported on the graph of the fractal transformation.  
	\begin{proposition}
		Let $\mathcal{F}=\{\mathbb{R}^d; f_1,f_2,\cdots ,f_N\}$ and  $\mathcal{G}=\{\mathbb{R}^d; g_1,g_2,\cdots,g_N\}$ be two IFSs such that $f_i$ and $g_i$ are bi-Lipschitz mappings as follows
		$$c_i \|x-x'\|\leq \|f_i(x)-f_i(x')\|\leq c_i' \|x-x'\|,$$
		$$r_i \|y-y'\|\leq \|g_i(y)-g_i(y')\|\leq r_i' \|y-y'\|,$$
		where $c_i,c_i',r_i,r_i'\in \mathbb{R}$, $0<c_i\leq c_i'<1$ and $0<r_i\leq r_i'<1.$ Assume that $\mathcal{F}$ satisfies SOSC and $\mathcal{C}_\mathcal{F}\prec \mathcal{C}_\mathcal{G}.$ Let $\mathcal{H}=\{\mathbb{R}^d\times \mathbb{R}^d; h_1,h_2,\cdots,h_N\}$ be the IFS with probability vector $(p_1,p_2,\cdots,p_N),$ where $h_i(x,y)=(f_i(x),g_i(y)).$ Let $\mu$ be the invariant Borel probability measure corresponding to the IFS $\mathcal{H}$ such that $\mu(T_\mathcal{FG})=1$. Then 
		$$s_r\leq \underline{D}_r(\mu)\leq \overline{D}_r(\mu)\leq t_r,$$
		where $s_r$ and $ t_r$ are uniquely given by $\sum\limits_{i=1}^{N}(p_i\min\{c_i,r_i\}^r)^\frac{s_r}{r+s_r}=1$ and  $\sum\limits_{i=1}^{N}(p_i\max\{c_i',r_i'\}^r)^\frac{t_r}{r+t_r}=1$, respectively.
	\end{proposition}
	\begin{proof}
		Since $\mathcal{F}$ satisfies SOSC and $\mathcal{C}_\mathcal{F}\prec \mathcal{C}_\mathcal{G}$, by using \cite[Theorem 6.8]{And}, the graph of the fractal transformation $T_\mathcal{FG}$ is same as the attractor of the IFS $\mathcal{H}.$ Using Theorem \ref{thmbilip}, part $(3)$ of Lemma \ref{Cond} and part $(2)$ of Lemma \ref{sibi}, we get our required result.
	\end{proof}


	\bibliographystyle{amsplain}

\end{document}